\documentclass[a4paper,11pt,reqno]{amsart}

\usepackage[utf8x]{inputenc}
\usepackage{hyperref}
\usepackage{xcolor}
\hypersetup{
    colorlinks,
    linkcolor={red!50!black},
    citecolor={blue!50!black},
    urlcolor={blue!80!black}
}

\usepackage{microtype}
\usepackage{textcomp}
\usepackage{MnSymbol}

\usepackage[mathcal]{euscript}
\let\mathcaltmp\mathcal
\usepackage{mathrsfs}
\let\mathcal\mathscr
\let\mathscr\mathcaltmp

\newtheoremstyle{plain}
  {.5\baselineskip\@plus.2\baselineskip\@minus.2\baselineskip}
  {.5\baselineskip\@plus.2\baselineskip\@minus.2\baselineskip\@plus.5em}
  {\slshape}
  {}
  {\bfseries}
  {.}
  { }
  {}

\newtheoremstyle{definition}
  {.5\baselineskip\@plus.2\baselineskip\@minus.2\baselineskip}
  {0.8\baselineskip\@plus.2\baselineskip\@minus.2\baselineskip\@plus.5em}
  {}
  {}
  {\bfseries}
  {.}
  { }
  {}

\makeatletter
\newcommand{\eqnum}{\refstepcounter{equation}\textup{\tagform@{\theequation}}}
\makeatother

\makeatletter \@addtoreset{equation}{section} \makeatother
\renewcommand{\theequation}{\thesection.\arabic{equation}}

\newtheorem{thm}[equation]{Theorem}

\newtheorem*{thm*}{Theorem}
\newtheorem{lem}[equation]{Lemma}
\newtheorem{cor}[equation]{Corollary}
\newtheorem{prop}[equation]{Proposition}
\newtheorem*{defthm*}{Definition/Theorem}

\theoremstyle{definition}
\newtheorem{defn}[equation]{Definition}
\newtheorem{rem}[equation]{Remark}
\newtheorem{exam}[equation]{Example}

\newtheorem{notat}[equation]{Notation}

\newtheorem*{exam*}{Example}

\newcommand\arXiv[1]{\href{http://arxiv.org/abs/#1}{arXiv:#1}}

\usepackage{xparse}

\usepackage{xspace}

\newcommand{\changelocaltocdepth}[1]{%
  \addtocontents{toc}{\protect\setcounter{tocdepth}{#1}}%
  \setcounter{tocdepth}{#1}}

\newcommand{\nc}{\newcommand}
\nc{\renc}{\renewcommand}
\nc{\ssec}{\subsection}
\nc{\sssec}{\subsubsection}
\nc{\on}{\operatorname}
\nc{\term}[1]{#1\xspace}

\usepackage{tikz}
\usetikzlibrary{matrix}
\usepackage{tikz-cd}

\tikzset{
  commutative diagrams/.cd,
  arrow style=tikz,
  diagrams={>=latex}}
\makeatletter
\tikzset{
  column sep/.code=\def\pgfmatrixcolumnsep{\pgf@matrix@xscale*(#1)},
  row sep/.code   =\def\pgfmatrixrowsep{\pgf@matrix@yscale*(#1)},
  matrix xscale/.code=%
    \pgfmathsetmacro\pgf@matrix@xscale{\pgf@matrix@xscale*(#1)},
  matrix yscale/.code=%
    \pgfmathsetmacro\pgf@matrix@yscale{\pgf@matrix@yscale*(#1)},
  matrix scale/.style={/tikz/matrix xscale={#1},/tikz/matrix yscale={#1}}}
\def\pgf@matrix@xscale{1}
\def\pgf@matrix@yscale{1}
\makeatother

\usepackage[cmtip, all]{xy}

\usepackage[inline]{enumitem}
\setlist[enumerate,1]{label={(\alph*)},itemsep=\parskip}

\newlist{enumcompress}{enumerate}{1}
\setlist[enumcompress,1]{
  label={(\alph*)},
  itemsep=0.3\parskip,
  leftmargin=*,
  align=left,
  topsep=0em
}

\newlist{thmlist}{enumerate}{2}
\setlist[thmlist,1]{
  label={\em(\roman*)}, ref={(\roman*)},
  itemsep=0.5em,
  align=right,widest=vi)}
\setlist[thmlist,2]{
  label={\em(\alph*)}, ref={(\alph*)},
  itemsep=0.75em,
  labelsep=0em,labelindent=0em,leftmargin=*,align=left,widest=vi),
  topsep=0.75em}

\newlist{thmlistbis}{enumerate}{1}
\setlist[thmlistbis,1]{
  label={\em(\roman*~\textit{bis})},
  ref={(\roman*}~\textit{bis}\upshape{)},
  itemsep=0.5em,
  align=right, widest=vi)}

\newlist{defnlist}{enumerate}{2}
\setlist[defnlist,1]{
  label={(\roman*)}, ref={(\roman*)},
  itemsep=0.5em,
  align=right, widest=vi)}

\setlist[defnlist,2]{
  label={(\alph*)}, ref={(\alph*)},
  itemsep=0.75em,
  labelsep=0em,labelindent=0em,leftmargin=*,align=left,widest=vi),
  topsep=0.75em}

\newlist{inlinelist}{enumerate*}{1}
\setlist[inlinelist,1]{label={(\alph*)}}

\newlist{inlinedefnlist}{enumerate*}{1}
\definecolor{green}{HTML}{38550C}
\setlist[inlinedefnlist,1]{label={\color{green}(\roman*)}}

\nc{\cA}{\ensuremath{\mathcal{A}}\xspace}
\nc{\cB}{\ensuremath{\mathcal{B}}\xspace}
\nc{\cC}{\ensuremath{\mathcal{C}}\xspace}
\nc{\cD}{\ensuremath{\mathcal{D}}\xspace}
\nc{\cE}{\ensuremath{\mathcal{E}}\xspace}
\nc{\cF}{\ensuremath{\mathcal{F}}\xspace}
\nc{\cG}{\ensuremath{\mathcal{G}}\xspace}
\nc{\cH}{\ensuremath{\mathcal{H}}\xspace}
\nc{\cI}{\ensuremath{\mathcal{I}}\xspace}
\nc{\cJ}{\ensuremath{\mathcal{J}}\xspace}
\nc{\cK}{\ensuremath{\mathcal{K}}\xspace}
\nc{\cL}{\ensuremath{\mathcal{L}}\xspace}
\nc{\cM}{\ensuremath{\mathcal{M}}\xspace}
\nc{\cN}{\ensuremath{\mathcal{N}}\xspace}
\nc{\cO}{\ensuremath{\mathcal{O}}\xspace}
\nc{\cP}{\ensuremath{\mathcal{P}}\xspace}
\nc{\cQ}{\ensuremath{\mathcal{Q}}\xspace}
\nc{\cR}{\ensuremath{\mathcal{R}}\xspace}
\nc{\cS}{\ensuremath{\mathcal{S}}\xspace}
\nc{\cT}{\ensuremath{\mathcal{T}}\xspace}
\nc{\cU}{\ensuremath{\mathcal{U}}\xspace}
\nc{\cV}{\ensuremath{\mathcal{V}}\xspace}
\nc{\cW}{\ensuremath{\mathcal{W}}\xspace}
\nc{\cX}{\ensuremath{\mathcal{X}}\xspace}
\nc{\cY}{\ensuremath{\mathcal{Y}}\xspace}
\nc{\cZ}{\ensuremath{\mathcal{Z}}\xspace}

\nc{\sA}{\ensuremath{\mathscr{A}}\xspace}
\nc{\sB}{\ensuremath{\mathscr{B}}\xspace}
\nc{\sC}{\ensuremath{\mathscr{C}}\xspace}
\nc{\sD}{\ensuremath{\mathscr{D}}\xspace}
\nc{\sE}{\ensuremath{\mathscr{E}}\xspace}
\nc{\sF}{\ensuremath{\mathscr{F}}\xspace}
\nc{\sG}{\ensuremath{\mathscr{G}}\xspace}
\nc{\sH}{\ensuremath{\mathscr{H}}\xspace}
\nc{\sI}{\ensuremath{\mathscr{I}}\xspace}
\nc{\sJ}{\ensuremath{\mathscr{J}}\xspace}
\nc{\sK}{\ensuremath{\mathscr{K}}\xspace}
\nc{\sL}{\ensuremath{\mathscr{L}}\xspace}
\nc{\sM}{\ensuremath{\mathscr{M}}\xspace}
\nc{\sN}{\ensuremath{\mathscr{N}}\xspace}
\nc{\sO}{\ensuremath{\mathscr{O}}\xspace}
\nc{\sP}{\ensuremath{\mathscr{P}}\xspace}
\nc{\sQ}{\ensuremath{\mathscr{Q}}\xspace}
\nc{\sR}{\ensuremath{\mathscr{R}}\xspace}
\nc{\sS}{\ensuremath{\mathscr{S}}\xspace}
\nc{\sT}{\ensuremath{\mathscr{T}}\xspace}
\nc{\sU}{\ensuremath{\mathscr{U}}\xspace}
\nc{\sV}{\ensuremath{\mathscr{V}}\xspace}
\nc{\sW}{\ensuremath{\mathscr{W}}\xspace}
\nc{\sX}{\ensuremath{\mathscr{X}}\xspace}
\nc{\sY}{\ensuremath{\mathscr{Y}}\xspace}
\nc{\sZ}{\ensuremath{\mathscr{Z}}\xspace}

\nc{\bA}{\ensuremath{\mathbf{A}}\xspace}
\nc{\bB}{\ensuremath{\mathbf{B}}\xspace}
\nc{\bC}{\ensuremath{\mathbf{C}}\xspace}
\nc{\bD}{\ensuremath{\mathbf{D}}\xspace}
\nc{\bE}{\ensuremath{\mathbf{E}}\xspace}
\nc{\bF}{\ensuremath{\mathbf{F}}\xspace}
\nc{\bG}{\ensuremath{\mathbf{G}}\xspace}
\nc{\bH}{\ensuremath{\mathbf{H}}\xspace}
\nc{\bI}{\ensuremath{\mathbf{I}}\xspace}
\nc{\bJ}{\ensuremath{\mathbf{J}}\xspace}
\nc{\bK}{\ensuremath{\mathbf{K}}\xspace}
\nc{\bL}{\ensuremath{\mathbf{L}}\xspace}
\nc{\bM}{\ensuremath{\mathbf{M}}\xspace}
\nc{\bN}{\ensuremath{\mathbf{N}}\xspace}
\nc{\bO}{\ensuremath{\mathbf{O}}\xspace}
\nc{\bP}{\ensuremath{\mathbf{P}}\xspace}
\nc{\bQ}{\ensuremath{\mathbf{Q}}\xspace}
\nc{\bR}{\ensuremath{\mathbf{R}}\xspace}
\nc{\bS}{\ensuremath{\mathbf{S}}\xspace}
\nc{\bT}{\ensuremath{\mathbf{T}}\xspace}
\nc{\bU}{\ensuremath{\mathbf{U}}\xspace}
\nc{\bV}{\ensuremath{\mathbf{V}}\xspace}
\nc{\bW}{\ensuremath{\mathbf{W}}\xspace}
\nc{\bX}{\ensuremath{\mathbf{X}}\xspace}
\nc{\bY}{\ensuremath{\mathbf{Y}}\xspace}
\nc{\bZ}{\ensuremath{\mathbf{Z}}\xspace}

\nc{\bbA}{\ensuremath{\mathbb{A}}\xspace}
\nc{\bbB}{\ensuremath{\mathbb{B}}\xspace}
\nc{\bbC}{\ensuremath{\mathbb{C}}\xspace}
\nc{\bbD}{\ensuremath{\mathbb{D}}\xspace}
\nc{\bbE}{\ensuremath{\mathbb{E}}\xspace}
\nc{\bbF}{\ensuremath{\mathbb{F}}\xspace}
\nc{\bbG}{\ensuremath{\mathbb{G}}\xspace}
\nc{\bbH}{\ensuremath{\mathbb{H}}\xspace}
\nc{\bbI}{\ensuremath{\mathbb{I}}\xspace}
\nc{\bbJ}{\ensuremath{\mathbb{J}}\xspace}
\nc{\bbK}{\ensuremath{\mathbb{K}}\xspace}
\nc{\bbL}{\ensuremath{\mathbb{L}}\xspace}
\nc{\bbM}{\ensuremath{\mathbb{M}}\xspace}
\nc{\bbN}{\ensuremath{\mathbb{N}}\xspace}
\nc{\bbO}{\ensuremath{\mathbb{O}}\xspace}
\nc{\bbP}{\ensuremath{\mathbb{P}}\xspace}
\nc{\bbQ}{\ensuremath{\mathbb{Q}}\xspace}
\nc{\bbR}{\ensuremath{\mathbb{R}}\xspace}
\nc{\bbS}{\ensuremath{\mathbb{S}}\xspace}
\nc{\bbT}{\ensuremath{\mathbb{T}}\xspace}
\nc{\bbU}{\ensuremath{\mathbb{U}}\xspace}
\nc{\bbV}{\ensuremath{\mathbb{V}}\xspace}
\nc{\bbW}{\ensuremath{\mathbb{W}}\xspace}
\nc{\bbX}{\ensuremath{\mathbb{X}}\xspace}
\nc{\bbY}{\ensuremath{\mathbb{Y}}\xspace}
\nc{\bbZ}{\ensuremath{\mathbb{Z}}\xspace}

\nc{\mrm}[1]{\ensuremath{\mathrm{#1}}\xspace}
\nc{\mit}[1]{\ensuremath{\mathit{#1}}\xspace}
\nc{\mbf}[1]{\ensuremath{\mathbf{#1}}\xspace}
\nc{\mcal}[1]{\ensuremath{\mathcal{#1}}\xspace}
\nc{\msc}[1]{\ensuremath{\mathscr{#1}}\xspace}

\nc{\sub}{\subseteq}
\nc{\too}{\longrightarrow}
\nc{\hook}{\hookrightarrow}
\nc{\hooklongrightarrow}{\lhook\joinrel\longrightarrow}
\nc{\hooklong}{\hooklongrightarrow}
\nc{\hooklongleftarrow}{\longleftarrow\joinrel\rhook}
\nc{\twoheadlongrightarrow}{\relbar\joinrel\twoheadrightarrow}
\nc{\longrightleftarrows}{\ \raisebox{0.3ex}{\(\mathrel{\substack{\xrightarrow{\rule{1em}{0em}} \\[-1ex] \xleftarrow{\rule{1em}{0em}}}}\)}\ }

\renc{\ge}{\geqslant}
\renc{\le}{\leqslant}

\nc{\id}{\mathrm{id}}

\DeclareMathOperator{\rk}{\mathrm{rk}}
\DeclareMathOperator{\Hom}{\on{Hom}}
\nc{\uHom}{\underline{\smash{\Hom}}}

\DeclareMathOperator{\End}{\on{End}}
\DeclareMathOperator{\Sym}{\on{Sym}}
\nc{\uEnd}{\underline{\smash{\End}}}

\nc{\colim}{\varinjlim}
\renc{\lim}{\varprojlim}
\nc{\Cofib}{\on{Cofib}}
\nc{\Fib}{\on{Fib}}
\nc{\initial}{\varnothing}
\nc{\op}{\mathrm{op}}

\DeclareMathOperator*{\fibprod}{\times}

\renc{\setminus}{\smallsetminus}

\usepackage{mathtools}
\newcommand\restr[2]{{
  \left.\kern-\nulldelimiterspace 
  #1 
  \vphantom{\big|} 
  \right|_{#2} 
  }}

\newcommand{\thmref}[1]{Theorem~\ref{#1}}

\newcommand{\secref}[1]{Sect.~\ref{#1}}
\newcommand{\ssecref}[1]{Subsect. ~\ref{#1}}

\newcommand{\lemref}[1]{Lemma~\ref{#1}}
\newcommand{\propref}[1]{Proposition~\ref{#1}}
\newcommand{\corref}[1]{Corollary~\ref{#1}}

\renewcommand{\eqref}[1]{(\ref{#1})}

\newcommand{\examref}[1]{Example~\ref{#1}}

\newcommand{\itemref}[1]{\ref{#1}}

\nc{\A}{\bA}
\renc{\P}{\bP}
\nc{\V}{\bV}
\nc{\Spec}{\on{Spec}}
\nc{\cl}{\mrm{cl}}
\nc{\D}{\on{\mbf{D}}}
\nc{\Ex}{\mrm{Ex}}
\nc{\gys}{\mrm{gys}}
\nc{\un}{\mbf{1}}
\nc{\vb}[1]{\langle{#1}\rangle}
\nc{\pr}{\mrm{pr}}
\nc{\unit}{\mrm{unit}}
\nc{\counit}{\mrm{counit}}
\nc{\DVect}{\on{DVect}}
\nc{\Perf}{\on{Perf}}
\nc{\SH}{\on{\mbf{SH}}}
\nc{\DM}{\on{\mbf{DM}}}
\nc{\MGL}{\mrm{MGL}}
\nc{\MGLmod}{\on{\mbf{D}_\MGL}}
\nc{\Nis}{\mrm{Nis}}
\nc{\et}{\mrm{\acute{e}t}}

\nc{\FT}{\on{FT}}
\nc{\Gm}{{\bG_{m}}}
\nc{\ev}{\mrm{ev}}
\nc{\coev}{\mrm{coev}}
\nc{\act}{\mrm{act}}
\nc{\tw}[1]{\{#1\}}
\nc{\Chom}{\on{C}_\bullet}
\nc{\BM}{{\mrm{BM}}}
\nc{\oH}{{\mrm{H}}}

\newcommand{\quo}[1]{{{}^\circ {#1}}}

\nc{\inftyCat}{\term{$\infty$-category}}
\nc{\inftyCats}{\term{$\infty$-categories}}

\nc{\inftyGrpd}{\term{$\infty$-groupoid}}
\nc{\inftyGrpds}{\term{$\infty$-groupoids}}

\title{The derived homogeneous Fourier transform}
\author[A.\,A. Khan]{Adeel A. Khan}
\date{2024-09-25}

\makeatletter
\def\l@subsection{\@tocline{2}{0pt}{4pc}{6pc}{}}
\makeatother

\begin{document}

\begin{abstract}
  We study a derived version of Laumon's homogeneous Fourier transform, which exchanges $\Gm$-equivariant sheaves on a derived vector bundle and its dual.
  In this context, the Fourier transform exhibits a duality between derived and stacky phenomena.
  This is the first in a series of papers on derived microlocal sheaf theory.
\end{abstract}

\maketitle

\renewcommand\contentsname{\vspace{-1cm}}
\tableofcontents

\setlength{\parindent}{0em}
\parskip 0.75em


\changelocaltocdepth{1}

\section*{Introduction}

Let $E$ be a vector bundle over a scheme $S$.
Laumon \cite{Laumon} introduced a geometric Fourier transform
\begin{equation*}
  \D^{\Gm}(E; \overline{\bQ}_\ell)
  \to \D^{\Gm}(E^\vee; \overline{\bQ}_\ell)
\end{equation*}
on homogeneous (= $\Gm$-equivariant) bounded constructible derived categories of $\ell$-adic sheaves.
It can be regarded as a uniform version of the $\ell$-adic Fourier transform of Deligne (for base fields of characteristic $p>0$) and the $D$-module 
Fourier transform of Brylinski--Malgrange--Verdier (for base fields of characteristic $0$).

In this paper we describe an extension of this construction from vector bundles to \emph{derived vector bundles}, i.e., total spaces of perfect complexes.
The total space of a perfect complex exists naturally as a derived Artin stack, which can exhibit both derived and stacky behaviour depending on the amplitude of the complex.
For a complex of vector bundles concentrated in nonpositive degrees, its realization as a derived vector bundle is non-stacky but singular and derived; for a complex in nonnegative degrees it is smooth and underived but stacky.

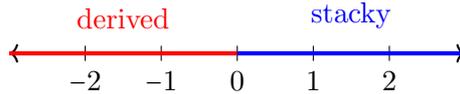
\begin{center}
  \begin{tikzpicture}
    \draw[<->, thick] (-3,0) -- (3,0);
    \foreach \x in {-2,-1,0,1,2}
        \draw (\x,0.1) -- (\x,-0.1) node[below] {$\x$};

    \draw[red, ultra thick] (-3,0) -- (0,0);
    \node[above, red] at (-1.5,0.2) {derived};

    \draw[blue, ultra thick] (0,0) -- (3,0);
    \node[above, blue] at (1.5,0.2) {stacky};
  \end{tikzpicture}
\end{center}

So, in this context, the Fourier transform manifests a duality between derived and stacky phenomena.
What we will see is that, perhaps surprisingly, the fundamental properties of the homogeneous Fourier transform persist to the derived setting.
Most importantly, we have the involutivity property: for every derived vector bundle $E$ there is a canonical isomorphism
\begin{equation}
  \FT_{E^\vee} \circ \FT_E \simeq \id
\end{equation}
up to a twist.

For classical vector bundles, the proof of involutivity rests on the computation
\begin{equation}
  \FT_{E^\vee}(\bQ_{\ell})
  \simeq 0_{E,*}(\bQ_{\ell})(-r)[-2r]
\end{equation}
where $0_E : S \to E$ is the zero section and $r=\rk(E)$; in particular, the \emph{Fourier transform of the constant sheaf is supported on the zero section}.
For an arbitrary derived vector bundle $E$, where $0_E$ need not be a closed immersion (unless $E$ is of nonpositive amplitude), the functors $0_{E,*}$ and $0_{E,!}$ give two possible generalizations of this.
It turns out that the version that is both correct and relevant is the $!$-version: the canonical morphism
\begin{equation*}
  0_{E,!} 0_E^! (\FT_{E^\vee}(\bQ_{\ell})) \to \FT_{E^\vee}(\bQ_{\ell})
\end{equation*}
is always invertible, regardless of the amplitude of $E$.
We call this the \emph{cosupport} property; it is the key technical insight in this paper, from which involutivity and all other desired properties of $\FT_E$ follow.

Along the way we will also show that the homogeneous Fourier transform is well-behaved in the context of any reasonable six functor formalism.
More precisely, we will work in the generality of any topological weave \cite{Weaves}.
This includes classical six functor formalisms such as the derived \inftyCat of sheaves of abelian groups (over $\bC$) or the derived \inftyCat of $\ell$-adic sheaves (over a base on which $\ell$ is invertible), but also various motivic \inftyCats: Voevodsky motives, MGL-modules, or motivic spectra (not even orientability is required).

One consequence of derived Fourier duality is the existence of a canonical isomorphism in $\Gm$-equivariant Borel--Moore homology
\begin{equation}
  \oH^{\BM,\Gm}_*(E)
  \simeq \oH^{\BM,\Gm}_{*-2r}(E^\vee[-1])
\end{equation}
for every derived vector bundle $E \to S$ of rank $r$ (see \corref{cor:fouriso}).
This generalizes Kashiwara's Fourier isomorphism (\examref{exam:subtext}).
Following \cite{MirkovicRicheChern}, one expects that (when $S$ is a quasi-smooth derived scheme and $E$ is of amplitude $[-1,0]$) this is a decategorification of the \emph{linear Koszul duality} equivalence between the stable \inftyCats of $\Gm$-equivariant ind-coherent sheaves on $E$ and $E^\vee[-1]$ (see \cite{MirkovicRiche}, \cite[\S 4.6]{TodaCatLocal}).

The derived Fourier transform was conceived in late 2019, as part of a program on derived microlocal sheaf theory, and first announced in a talk at Kavli IPMU in February 2020.
The first version of this paper was written in January 2021, but the current proof using the cosupport property (whence the path to generalizing beyond the case discussed in \propref{prop:gluetwist}) was only discovered recently.
The argument using the cosupport property can also be applied to other sheaf-theoretic Fourier transforms, such as the version for monodromic sheaves in \cite{Wang}.
A derived version of Deligne's original $\ell$-adic Fourier transform is recently considered in \cite{FYZ}; our arguments give a simpler proof in that context as well.
They also apply to the derived version of the \emph{Fourier--Sato} transform \cite{KashiwaraSchapira}, which we consider in forthcoming joint work with T.~Kinjo \cite{dimredcoha}.

\ssec{Outline}

We begin in \secref{sec:main} by defining the homogeneous Fourier transform for derived vector bundles and stating our results about it.
In \secref{sec:easy} we prove the easier properties that are independent of involutivity.
Sections~\ref{sec:cosuppinvol}, \ref{sec:supp}, and \ref{sec:cosupp} achieve the proof of involutivity through the cosupport property.
Finally the remaining properties of the Fourier transform are derived from involutivity in \secref{sec:funct2}.
Appendix~\ref{sec:equivariant} recalls the contraction lemma and its consequences in our context.
In Appendix~\ref{sec:j} we redo some straightforward computations from \cite{Laumon} that don't involve derived vector bundles, adapting the arguments from \emph{op. cit.} to the generality we work in here.

\ssec{Acknowledgments}

I would like to Denis-Charles Cisinski for introducing me to sheaf-theoretic Fourier transforms and his suggestion to consider its motivic analogue during a seminar we coorganized in Fall 2018.
Thanks also to Charanya Ravi and Masoud Zargar for discussions during that seminar.
I am extremely grateful to Tasuki Kinjo for invaluable discussions about derived Fourier transforms and for his collaboration on derived microlocal sheaf theory in the complex-analytic context \cite{dimredcoha}; thanks to him in particular for pointing out the sign issue in \propref{prop:gluetwist}.
I also thank Tony Feng for his interest, questions, and feedback on the 2021 version of this paper.
Finally, I thank RIMS/Kyoto University for the ideal working conditions, where the new proof using the cosupport property was conceived.

I acknowledge support from the grants AS-CDA-112-M01 (Academia Sinica), NSTC 110-2115-M-001-016-MY3, and NSTC 112-2628-M-001-0062030.

\ssec{Conventions and notation}

  \sssec{Sheaves}

    Throughout the paper we fix a \emph{topological weave}, which is an axiomatization of a sheaf theory with the six functor formalism introduced in \cite{Weaves}.
    Roughly speaking this is a six-functor formalism with the following properties:
    \begin{enumcompress}
      \item\label{item:weave/loc}
      \emph{Localization:}
      The \inftyCat $\D(\initial)$ is zero.
      For any derived Artin stack $X$ and any closed-open decomposition $i : Z \to X$, $j : X\setminus Z \to X$, there is a canonical exact triangle of functors
      \begin{equation}\label{eq:loc}
        j_! j^* \to \id \to i_*i^*.
      \end{equation}

      \item\label{item:weave/htp}
      \emph{Homotopy invariance:}
      For any derived Artin stack $X$ and vector bundle $\pi : E \to X$, the unit morphism $\id \to \pi_*\pi^*$ is invertible.
    \end{enumcompress}
    Note that localization implies \emph{derived invariance}: for a derived Artin stack $X$, the inclusion of the classical truncation $X_\cl \hook X$ induces an equivalence $\D(X) \simeq \D(X_\cl)$ which commutes with each of the six functors.
    We also have \emph{Poincaré duality}, which for a topological weave gives a canonical isomorphism
    \begin{equation}\label{eq:Poin}
      f^!(-) \simeq f^*(-)\vb{T_f}
    \end{equation}
    for $f$ any smooth morphism with relative tangent bundle $f$.
    If $\D$ admits an orientation (as in the first five examples below), we may identify $\vb{T_f} \simeq \vb{d} := (d)[2d]$ where $d$ is the relative dimension of $f$.

    On (derived) schemes, examples of weaves are as follows:
    \begin{enumcompress}
      \item\emph{Betti sheaves (over $\bC$):}
      The assignment $X \mapsto \D(X)$ sending $X$ to the derived \inftyCat $\on{D}(X(\bC), \bZ)$ of sheaves of abelian groups on the topological space $X(\bC)$.

      \item\emph{Étale sheaves (finite coefficients, over $k$ with $n \in k^\times$):}
      The assignment $X \mapsto \D(X)$ sending $X$ to the derived \inftyCat $\on{D}_\et(X, \bZ/n\bZ)$ of sheaves of $\bZ/n\bZ$-modules on the small étale site of $X$.

      \item\emph{Étale sheaves ($\ell$-adic coefficients, over $k$ with $\ell \in k^\times$):}
      The assignment $X \mapsto \D(X)$ sending $X$ to the $\ell$-adic derived \inftyCat $\on{D}_\et(X, \bZ_\ell)$ of sheaves on the small étale site of $X$, i.e., the limit $\on{D}_\et(X, \bZ/\ell^n\bZ)$ over $n>0$.

      \item\label{item:weave/motives}\emph{Motives:}
      The assignment $X \mapsto \D(X)$ sending $X$ to the \inftyCat $\on{DM}(X) := \on{D_{H\bZ}}(X)$ of modules over the integral motivic Eilenberg--MacLane spectrum $H\bZ_X$.

      \item\emph{Cobordism motives:}
      The assignment $X\mapsto \D(X)$ sending $X$ to the \inftyCat $\MGLmod(X)$ of modules over Voevodsky's algebraic cobordism spectrum $\MGL_X$.

      \item\emph{Motivic spectra:}
      The assignment $X\mapsto \D(X)$ sending $X$ to the \inftyCat $\SH(X)$ of motivic spectra over $X$.
    \end{enumcompress}

  \sssec{Sheaves on stacks}

    When the weave satisfies étale descent, as in the first three examples, it can be extended to (derived) Artin stacks following the method of \cite{LiuZheng} (see also \cite[App.~A]{KhanVirtual}).
    In general, $\D$ only satisfies Nisnevich descent, so we need to restrict our attention to $\Nis$-Artin stacks (see \cite[\S 4]{Weaves}, \cite[\S 1]{equilisse}).
    For $\tau\in\{\Nis,\et\}$ we define $(\tau,n)$-Artin and $\tau$-Artin stacks as in \cite[0.2.2]{equilisse}:
    \begin{defnlist}
      \item
      A stack is $(\et,0)$-Artin, resp. $(\Nis,0)$-Artin, if it is an algebraic space (resp. a decent algebraic space).
      Here an algebraic space is \emph{decent} if it is Zariski-locally quasi-separated, or equivalently Nisnevich-locally a scheme.
      
      \item
      For $n>0$, $X$ is \emph{$(\tau,n)$-Artin} if it has $(\tau,n-1)$-representable diagonal and admits a smooth morphism $U \twoheadrightarrow X$ with $\tau$-local sections from some scheme $U$.
      A stack is \emph{$\tau$-Artin} if it is $(\tau,n)$-Artin for some $n$.
    \end{defnlist}
    For $\tau=\et$, these are the usual notions of $n$-Artin stacks and Artin stacks, while e.g. $(\Nis,1)$-Artin stacks are the same as quasi-separated $1$-Artin stacks with separated diagonal.

    If our chosen weave $\D$ does not satisfy étale descent, then we adopt the convention that \emph{Artin} means ``$\Nis$-Artin''.

  \sssec{\texorpdfstring{$\bG_m$}{Gm}-Equivariant sheaves}

    If $X$ is a derived Artin stack with $\Gm$-action, we define the \inftyCat of $\Gm$-equivariant sheaves on $X$ as the \inftyCat of sheaves on the quotient stack:
    \begin{equation*}
      \D^{\Gm}(X) := \D([X/\Gm]).
    \end{equation*}
    There is a forgetful functor
    \begin{equation*}
      \D^{\Gm}(X) \to \D(X)
    \end{equation*}
    defined by $*$-inverse image along the (smooth) quotient morphism $X \twoheadrightarrow [X/\Gm]$.

    For any $\Gm$-equivariant morphism $f : X \to Y$ the four operations $f^*$, $f_*$, $f_!$, $f^!$ are defined on the $\Gm$-equivariant category using the induced morphism $[X/\Gm] \to [Y/\Gm]$.
    Since each operation commutes with smooth $*$-inverse image, they each commute with forgetting equivariance.

    Note that, when the weave $\D$ does not satisfy étale descent, we are implicitly using the fact that the group scheme $\Gm$ is ``special'', so that $X \twoheadrightarrow [X/\Gm]$ admits Nisnevich-local sections and hence the quotients are $\Nis$-Artin whenever $X$ is.

  \sssec{Derived vector bundles}

    Let $S$ be a derived Artin stack.
    Given a perfect complex $\sE \in \Perf(S)$, we denote by $\V(\sE)$ the stack of cosections of $\sE$, or equivalently sections of $\sE^\vee$.
    That is, given a derived scheme $T$ over $S$, the $T$-points of $\V(\sE)$ over $S$ are morphisms $\sE|_T \to \sO_T$ in $\Perf(T)$.
    This agrees with Grothendieck's convention for vector bundles.

    \begin{defn}
      The stable \inftyCat $\DVect(S)$ of \emph{derived vector bundles} over $S$ is the essential image of the fully faithful functor $\sE \mapsto \V(\sE)$ from $\Perf(S)^\op$ to the \inftyCat of derived stacks over $S$ with $\bG_m$-action.
    \end{defn}

    Since $\sE \mapsto \V(\sE)$ is contravariant, we will adopt the following conventions:
    \begin{enumerate}
      \item If $E = \V(\sE)$, we write $E[n] := \V(\sE[-n])$ for every integer $n$.
      \item If $E = \V(\sE)$, we say $E$ is \emph{of amplitude $\ge 0$} (resp. $\le 0$, $[a,b]$) if $\sE$ is of Tor-amplitude $\le 0$ (resp. $\ge 0$, $[-b, -a]$) in homological grading.
      (Since $\sE$ is perfect, $E$ is of some finite amplitude.)
    \end{enumerate}

    Now that these are in place, we will avoid making any further reference to the construction $\sE \mapsto \V(\sE)$ (so that it makes no difference whether we took $\sE$ or $\sE^\vee$).

    \begin{notat}
      Given a derived vector bundle $E$ over a derived Artin stack $S$, we denote by $\pi_E : E \to S$ the projection and $0_E : S \to E$ the zero section.
    \end{notat}

    \begin{notat}
      For every $E \in \DVect(S)$, it will be convenient to denote the quotient by the $\bG_m$-scaling action by:
      \begin{equation}
        \quo{E} := [E/\bG_m].
      \end{equation}
      Given a morphism of derived vector bundles $\phi : E' \to E$, we also write $\quo{\phi} : \quo{E'} \to \quo{E}$ for the induced morphism.
      We will also use the same notation for $\bG_m$-invariant subsets of $E$, e.g. $\quo{\bG_{m,S}} = S$.
      Note that points of $\quo{E}$ are pairs $(L,\phi)$ where $L$ is a line bundle and $\phi : L \to E^\vee$ is a morphism of derived vector bundles.
    \end{notat}

    If the weave $\D$ does not satisfy étale descent, we need the following:

    \begin{prop}
      Let $S$ be a derived stack and $E \in \DVect(S)$ a derived vector bundle over $S$.
      If $E$ is of amplitude $\le 0$, then the projection $\pi_E : E \to S$ is affine.
      If $E$ is of amplitude $\le n$, where $n>0$, then $\pi_E$ is $(n,\Nis)$-Artin.
    \end{prop}
    \begin{proof}
      The proof is the same as that of $n$-Artinness which is well-known.
      We recall it anyway for the reader's convenience.
      Suppose $E$ is of amplitude $\le n$ where $n=0$ (resp. $n>0$).
      It is enough to show that if $S$ is affine, then $E$ is affine (resp. $(n,\Nis)$-Artin).
      In the $n=0$ case we have $E = \V(\sE) \simeq \Spec(\Sym(\sE))$ since $\sE$ is connective.
      Thus assume $n>0$.

      Choose a presentation of $E$ as a chain complex of vector bundles and let $E_{\ge 0}$ and $E_{<0}$ denote the stupid truncations so that we have an exact triangle $E_{<0} \to E \to E_{\ge 0}$, or in other words a homotopy cartesian square
      \begin{equation*}
        \begin{tikzcd}
          E \ar{r}{i}\ar{d}
          & E_{\ge 0} \ar{d}
          \\
          S \ar{r}{0}
          & E_{<0}[1].
        \end{tikzcd}
      \end{equation*}
      Since $E_{<0}[1]$ is of amplitude $\le 0$ and hence affine, its zero section is a closed immersion and thus so is $i$.
      Replacing $E$ by $E_{\ge 0}$, we may therefore assume that $E$ is of amplitude $[0, n]$.
      There is a canonical morphism $E \to E_{n}[n]$ whose cofibre $E'$ is of amplitude $[0, n-1]$.
      Thus we have a homotopy cartesian square
      \begin{equation*}
        \begin{tikzcd}
          E \ar{r}\ar{d}
          & E_{n}[n] \ar{d}
          \\
          S \ar{r}{0}
          & E'.
        \end{tikzcd}
      \end{equation*}
      By induction on $n$ we may assume $E'$ is $(\Nis,n-1)$-Artin.
      It will then suffice to show that $E_{n}[n]$ is $(\Nis,n)$-Artin for every $n>0$ (because then $E$ is a fibred product of $(\Nis,n)$-Artin stacks).
      In fact, the zero section $0 : S \twoheadrightarrow E_{n}[n]$ is smooth and has a global section because $S$ is affine and $E_{n}[n]$ is of amplitude $>0$.
      Thus it is a $(\Nis,n)$-Artin atlas for $E_{n}[n]$.
    \end{proof}

    \begin{rem}
      In fact one can easily show: $E$ is of amplitude $\le 0$ if and only if $\pi_E$ is affine, if and only if $\pi_E$ is representable, if and only if $0_E$ is a closed immersion.
      Similarly, $E$ is of amplitude $\ge 0$ if and only if $\pi_E$ is smooth, if and only if $\pi_{E^\vee}$ is affine.
    \end{rem}

\changelocaltocdepth{2}


\section{Definition and properties}
\label{sec:main}
  
\ssec{Definition}
\label{ssec:main/def}

  Let $S$ be a derived Artin stack and $E \in \DVect(S)$ a derived vector bundle over $S$.

  \sssec{Kernel}

    Consider the homogeneous evaluation morphism
    \begin{equation}
      \ev_E : \quo{E^\vee} \fibprod_S \quo{E} \to \quo{\A^1_S}
    \end{equation}
    sending a pair
    \begin{equation*}
      \big((L, \phi : L \to E^\vee), (L', \psi : L' \to E)\big)
    \end{equation*}
    to
    \begin{equation*}
      \big( L \otimes L', L \otimes L' \xrightarrow{\phi \otimes \psi} E^\vee \otimes_S E \xrightarrow{\ev} \A^1_S \big).
    \end{equation*}
    We set
    \begin{equation}
      \sP_E := \ev_E^* \quo{j_*}(\un) \in \D(\quo{E^\vee} \fibprod_S \quo{E})
    \end{equation}
    where $j : \bG_{m,S} \to \A^1_S$ is the inclusion and $\quo{j} : S = \quo{\bG_{m,S}} \hook \quo{\A^1_S}$ is scaling quotient.

  \sssec{Transform}
  
    Let $\pr_1$ and $\pr_2$ denote the respective projections
    \begin{equation*}
      \begin{tikzcd}
        & {\quo{E^\vee} \fibprod_S \quo{E}} \ar{rd}{\pr_2}\ar[swap]{ld}{\pr_1} &
        \\
        {\quo{E^\vee}} & & {\quo{E}}
      \end{tikzcd}
    \end{equation*}

    The \emph{homogeneous Fourier transform} on $E$ is the functor
    \begin{equation*}
      \FT_E : \D^{\Gm}(E) \to \D^{\Gm}(E^\vee)
    \end{equation*}
    defined by
    \begin{equation}
      \sF \mapsto \pr_{1,!}(\pr_2^*(\sF) \otimes \sP_E)[-1].
    \end{equation}

\ssec{Zero bundle}
\label{ssec:main/zero}

  By abuse of notation, we denote the zero bundle $\V_S(0) = S$ by $0_S \in \DVect(S)$.

  \begin{prop}\label{prop:zero}
    Let $o : B\bG_{m,S} \to B\bG_{m,S}$ be the involution $L \mapsto L^\vee$.
    There are canonical isomorphisms $\FT_{0_S} \simeq o^* \simeq o_* \simeq o_! \simeq o^!$.
  \end{prop}

  See \ssecref{ssec:easy/zero}.

\ssec{Involutivity}
\label{ssec:main/invol}

  In the most general case, our statement of involutivity is up to twisting with a canonical $\otimes$-invertible object.

  \begin{lem}\label{lem:L}
    Let $E \in \DVect(S)$.
    The object
    \begin{equation}\label{eq:L}
      \sL^{E} :=
      \pi_{E,!} \FT_{E^\vee}(\un_{E^\vee})
      \simeq 0_{E}^! \FT_{E^\vee}(\un_{E^\vee})
      \in \D^{\Gm}(S)
    \end{equation}
    is $\otimes$-invertible.
  \end{lem}

  See \ssecref{ssec:cosupp/L}.

  \begin{thm}\label{thm:invol}
    For every $E \in \DVect(S)$, there is a canonical isomorphism
    \begin{equation}\label{eq:hygrothermal}
      (-) \otimes \pi_{E}^*(\sL^{E}) \to \FT_{E^\vee} \FT_{E}(-)
    \end{equation}
    of functors $\D^{\Gm}(E) \to \D^{\Gm}(E)$.
  \end{thm}

  See \secref{sec:cosupp}.

  \begin{notat}
    Let $E \in \DVect(S)$.
    Tensoring with $\sL^E$ defines an auto-equivalence
    \begin{equation*}
      (-)\tw{E} := (-) \otimes \sL^E
      \quad \text{of}~\D^{\Gm}(S).
    \end{equation*}
    We denote its inverse by $(-)\tw{-E} := (-) \otimes \sL^{E,\otimes -1}$.
    We also write $(-)\tw{E} := (-) \otimes f^*(\sL^E)$ of $\D^{\Gm}(X)$ where $f : X \to S$ is any derived Artin stack over $S$ with $\Gm$-action.
    In this notation \eqref{eq:hygrothermal} reads:
    \begin{equation}
      (-)\tw{E} \to \FT_{E^\vee} \FT_{E}(-).
    \end{equation}
  \end{notat}

  \begin{cor}\label{cor:inverse}
    For every $E \in \DVect(S)$, the functor $\FT_{E^\vee}(\bullet)\tw{-E}$ determines a canonical inverse to $\FT_E$.
  \end{cor}

\ssec{Base change}
\label{ssec:main/funbase}

  \begin{prop}\label{prop:funbase}
    For every morphism $f : S' \to S$, denote by $f_E : E' \to E$ and $f_{E^\vee} : E'^\vee \to E^\vee$ its base changes.
    Then there are canonical isomorphisms
    \begin{align}
      f_{E^\vee}^* \circ \FT_{E} &\simeq {\FT_{E'}} \circ f_E^*,\tag{BC$^*$}\label{eq:BC^*}\\
      f_{E^\vee,*} \circ {\FT_{E'}} &\simeq {\FT_{E}} \circ f_{E,*},\tag{BC$_*$}\label{eq:BC_*}\\
      f_{E^\vee,!} \circ \FT_{E'} &\simeq {\FT_{E}} \circ f_{E,!}\tag{BC$_!$}\label{eq:BC_!}\\
      f_{E^\vee}^! \circ \FT_{E} &\simeq {\FT_{E'}} \circ f_E^!\tag{BC$^!$}\label{eq:BC^!}.
    \end{align}
  \end{prop}

  See \ssecref{ssec:basechange1} and \secref{sec:funct2}.

  \begin{lem}\label{lem:twistbase}
    Let $E \in \DVect(S)$.
    For every morphism $f : S' \to S$, we have a canonical isomorphism
    \begin{equation}
      f^* \sL^E \simeq \sL^{E'}.
    \end{equation}
    Moreover, if $f_E : E' \to E$ denotes the base change, there are canonical isomorphisms
    \begin{align}
      f^*_E((-)\tw{E}) &\simeq f_E^*(-)\tw{E'},\\
      f^!_E((-)\tw{E}) &\simeq f_E^!(-)\tw{E'},\\
      f_{E,*}((-)\tw{E'}) &\simeq f_{E,*}(-)\tw{E},\\
      f_{E,!}((-)\tw{E'}) &\simeq f_{E,!}(-)\tw{E},
    \end{align}
    and similarly for $f_{E^\vee} : E'^\vee \to E^\vee$.
  \end{lem}

  See Subsects.~\ref{ssec:basechangetwist1} and \ref{ssec:basechangetwist2}.

\ssec{Functoriality}
\label{ssec:main/fun}

  \begin{prop}\label{prop:fun}
    For every morphism of derived vector bundles $\phi : E' \to E$ over $S$, there are canonical isomorphisms
    \begin{align}
      \Ex^{*,\FT} &: \phi^{\vee,*} \circ {\FT_{E'}} \to {\FT_{E}} \circ \phi_!\tag{Fun$^*$}\label{eq:funA}\\
      \Ex^{\FT,!} &: {\FT_{E'}} \circ \phi^! \to \phi^\vee_* \circ {\FT_{E}}\tag{Fun$_*$}\label{eq:funA'}\\
      \Ex^{!,\FT} &: \phi^{\vee,!} \circ {\FT_{E'}\tw{-E'}} \to \FT_{E}\tw{-E} \circ \phi_*\tag{Fun$^!$}\label{eq:funB}\\
      \Ex^{\FT,*} &: {\FT_{E'}\tw{-E'}} \circ \phi^* \to \phi^\vee_! \circ {\FT_{E}\tw{-E}} \tag{Fun$_!$}\label{eq:funB'}
    \end{align}
  \end{prop}

  See \ssecref{ssec:easy/fun} and \secref{sec:funct2}.

  \begin{exam}
    Let $E \in \DVect(S)$.
    Since the projection $\pi_{E} : E \to S$ and zero section $0_{E} : S \to E$ are dual to $0_{E^\vee} : S \to E^\vee$ and $\pi_{E^\vee} : E^\vee \to S$, respectively, we get canonical isomorphisms
    \begin{align}
      \FT_{0_S} \circ ~\pi_{E,!} &\simeq 0_{E^\vee}^* \circ \FT_{E}\\
      \FT_{E} \circ ~0_{E,!} &\simeq \pi_{E^\vee}^* \circ \FT_{0_S}.
    \end{align}
  \end{exam}

\ssec{Outline of proof: support and cosupport properties}
\label{ssec:main/cosupp}

  We will see that the proof of involutivity (\thmref{thm:invol}) eventually boils down to what we call the ``cosupport property'' for the object $\FT_{E^\vee}(\un_{E^\vee}) \in \D^{\Gm}(E)$.

  When $E$ is of amplitude $\le 0$, the object $\FT_{E^\vee}(\un_{E^\vee})$ is supported on the zero section (which is a closed immersion):

  \begin{prop}[Support property]\label{prop:suppinvol}
    Let $E \in \DVect(S)$.
    If $E$ is of amplitude $\le 0$, then we have:
    \begin{thmlist}
      \item\label{item:Ghent}
      There is a canonical isomorphism
      \begin{equation}
        {0}_E^* \FT_{E^\vee}(\un)
        \simeq \un_S\vb{-E^\vee}.
      \end{equation}

      \item\label{item:journalize}
      The canonical morphism
      \begin{equation}\label{eq:doctrinally}
        \FT_{E^\vee}(\un)
        \xrightarrow{\mrm{unit}} {0_{E,*}} {0_E^*} \FT_{E^\vee}(\un)
        \simeq {0_{E,*}}(\un_S)\vb{-E^\vee}
      \end{equation}
      is invertible.
    \end{thmlist}
    In particular, there is a canonical isomorphism
    \begin{equation}\label{eq:conntwist}
      \un_S\tw{E}
      \simeq 0_E^* \FT_{E^\vee}(\un)
      \simeq \un_S \vb{-E^\vee}.
    \end{equation}
  \end{prop}

  See \ssecref{ssec:supp/proof}.
  In general the zero section $0_E$ is not a closed immersion, so that $0_{E,!}$ does not agree with $0_{E,*}$.
  Nevertheless, the following dual version of \propref{prop:suppinvol} holds for $E$ of arbitrary amplitude:

  \begin{prop}[Cosupport property]\label{prop:main/cosupp}
    For every $E \in \DVect(S)$, the object $\FT_{E^\vee}(\un_{E^\vee}) \in \D^{\Gm}(E)$ lies in the essential image of the fully faithful functor $0_{E,!}$.
    Equivalently, the canonical morphism
    \begin{equation}\label{eq:urodelan}
      0_{E,!} (\un_S\tw{E})
      \simeq 0_{E,!} 0_{E}^!(\FT_{E^\vee}(\un_{E^\vee}))
      \xrightarrow{\counit} \FT_{E^\vee}(\un_{E^\vee})
    \end{equation}
    is invertible.
  \end{prop}

  See \ssecref{ssec:cosupp/proof}.
  Involutivity will then follow from:

  \begin{lem}\label{lem:cosuppinvol}
    Let $E \in \DVect(S)$.
    If $E$ satisfies the cosupport property, then there is a canonical isomorphism
    \begin{equation*}
      (-)\tw{E} \to \FT_{E^\vee} \FT_{E}(-).
    \end{equation*}
  \end{lem}

  See \secref{sec:cosuppinvol}.

\ssec{Identification of the twist}

  We do not know whether the twist $\sL^E \simeq \un_S\tw{E}$ can be identified with the Thom twist $\un_S\vb{-E^\vee}$ in general.
  If $E$ admits a global presentation as a chain complex of vector bundles, one can with some care build such an isomorphism from the vector bundle case.
  To glue together these local isomorphisms (choosing presentations locally on $S$) we would need to show they are compatible up to coherent homotopy.
  Assuming the existence of a suitable t-structure this question reduced to the heart, where we just need to check a cocycle condition.
  This line of argument leads to a proof of the following statement.
  
  \begin{prop}\label{prop:gluetwist}
    Suppose that the weave $\D$ admits an orientation and a t-structure in which the unit is discrete, i.e., $\un_S \in \D(S)^\heartsuit$ for all derived Artin stacks $S$.
    Then for every $E \in \DVect(S)$ there exists a canonical isomorphism
    \begin{equation}\label{eq:gluetwist}
      \sL^E \simeq \un_S\tw{E} \simeq \un_S\vb{-r}
    \end{equation}
    where $r=\rk(E)$.
  \end{prop}

  The gluing is more subtle than we have suggested.
  This is due to the fact that if $E$ is a vector bundle, the isomorphisms \eqref{eq:conntwist} for $E$ and $E^\vee$ only agree up to a sign.
  This sign needs to be taken into account into constructing the isomorphism \eqref{eq:gluetwist} even in the presence of a global resolution.
  In particular, when $E$ is of amplitude $\le 0$, \emph{the isomorphism of \eqref{eq:gluetwist} only agrees with that of \eqref{eq:conntwist} up to a sign} in general.
  We will give the details in a future revision (see also \cite[Lem.~2.20]{dimredcoha}).

\ssec{Fourier isomorphism in Borel--Moore homology}

  For simplicity we assume the weave $\D$ is oriented and admits a t-structure as in \propref{prop:gluetwist}.
  Given $E \in \DVect(S)$, which we regard as usual with the weight $1$ scaling $\Gm$-action, we write $E^\diamond$ to denote the same derived vector bundle but with $\Gm$-action by weight $-1$ scaling.

  \begin{thm}\label{thm:fouriso}
    Let $S$ be a derived Artin stack.
    For every derived vector bundle $E \in \DVect(S)$ of rank $r$, there is a canonical isomorphism
    \begin{equation*}
      \pi_{E,*} \pi_E^!(-)
      \simeq \pi_{E^\vee[-1]^\diamond,*} \pi_{E^\vee[-1]^\diamond}^!(-) \vb{r}
    \end{equation*}
    of functors $\D^{\Gm}(S) \to \D^{\Gm}(S)$.
  \end{thm}
  \begin{proof}
    By base change for the homotopy cartesian square
    \begin{equation*}
      \begin{tikzcd}
        E \ar{r}{\pi_E}\ar{d}{\pi_E}
        & S \ar{d}{0_{E[1]}}
        \\
        S \ar{r}{0_{E[1]}}
        & E[1],
      \end{tikzcd}
    \end{equation*}
    we have a canonical isomorphism
    \begin{equation*}
      \pi_{E,*} \pi_E^!(-)
      \simeq 0_{E[1]}^! 0_{E[1],*}(-).
    \end{equation*}
    By \eqref{eq:funA'}, \eqref{eq:funB}, and \propref{prop:gluetwist} we have
    \begin{align*}
      \FT_{0_S} 0_{E[1]}^! 0_{E[1],*}(-)
      &\simeq \pi_{E^\vee[-1],*} \FT_{E[1]} 0_{E[1],*}(-)\\
      &\simeq \pi_{E^\vee[-1],*} \pi_{E^\vee[-1]}^! \FT_{0_S}(-) \tw{E[1]}\\
      &\simeq \pi_{E^\vee[-1],*} \pi_{E^\vee[-1]}^! \FT_{0_S}(-) \vb{r}.
    \end{align*}
    Recall that $\FT_{0_S} \simeq o^! \simeq o_*$ where $o : B\bG_{m,S} \to B\bG_{m,S}$ is as in \propref{prop:zero}.
    We conclude that
    \begin{equation*}
      \pi_{E,*} \pi_E^!(-)
      \simeq o_* \pi_{E^\vee[-1],*} \pi_{E^\vee[-1]}^! o^!
      \simeq \pi_{E^\vee[-1]^\diamond,*} \pi_{E^\vee[-1]^\diamond}^!
    \end{equation*}
    as claimed.
  \end{proof}

  We can interpret this in Borel--Moore homology as follows.
  Fix a base commutative ring $k$.
  For a derived Artin stack $X$ with $\Gm$-action, locally of finite type over $k$, we define the complex of $\Gm$-equivariant Borel--Moore chains on $X$ by
  \begin{equation*}
    \Chom^{\BM,\Gm}(X) :=
    \Chom^{\BM}([X/\Gm]_{/B\Gm}) :=
    R\Gamma(B\Gm, \quo{a}_*\quo{a}^!(\un))
  \end{equation*}
  where $a : X \to \Spec(k)$ is the projection and $\quo{a} : [X/\Gm] \to B\Gm := [\Spec(k)/\Gm]$ is the induced morphism.

  \begin{cor}\label{cor:fouriso}
    In the situation of \thmref{thm:fouriso}, there is a canonical isomorphism
    \begin{equation*}
      \Chom^{\BM,\Gm}(E) \simeq \Chom^{\BM,\Gm}(E^\vee[-1])\vb{r}.
    \end{equation*}
  \end{cor}
  \begin{proof}
    It only remains to observe that $o$ induces an isomorphism
    \begin{equation*}
      \Chom^{\BM,\Gm}(E)
      \simeq \Chom^{\BM,\Gm}(E^\diamond)
    \end{equation*}
    for every $E \in \DVect(S)$.
  \end{proof}

  \begin{exam}\label{exam:subtext}
    Let $F_1$ and $F_2$ be vector subbundles of a vector bundle $V$ over a scheme $S$.
    If $E := F_1 \fibprod^R_V F_2 \in \DVect(S)$ then $E^\vee[-1] \simeq F_1^\perp \fibprod^R_{V^\vee} F_2^\perp$.
    Indeed, we have the commutative diagram of homotopy cartesian squares
    \begin{equation*}
      \begin{tikzcd}
        F_1^\perp \fibprod^R_{V^\vee} F_2^\perp \ar{r}\ar{d}
        & F_2^\perp \ar{r}\ar{d}
        & S \ar{d}{0}
        \\
        F_1^\perp \ar{r}\ar{d}
        & V^\vee \ar{r}\ar{d}
        & F_2^\vee \ar{d}
        \\
        S \ar{r}{0}
        & F_1^\vee \ar{r}
        & E^\vee.
      \end{tikzcd}
    \end{equation*}
    In this case, \corref{cor:fouriso} recovers Kashiwara's Fourier isomorphism (see \cite{EvensMirkovic}, \cite[\S 1.4]{MirkovicRicheChern}):
    \begin{equation*}
      \Chom^{\BM,\Gm}(F_1 \fibprod^R_V F_2)
      \simeq \Chom^{\BM,\Gm}(F_1^\perp \fibprod^R_{V^\vee} F_2^\perp)\vb{r}
    \end{equation*}
    where $r = \rk(E) = \rk(F_1)+\rk(F_2)-\rk(V)$.
  \end{exam}

  \begin{rem}
    Let $X$ be a derived Artin stack, locally homotopically of finite presentation over $k$.
    The Fourier isomorphism for the $1$-shifted tangent bundle reads:
    \begin{equation}\label{eq:lock}
      \oH^{\BM,\Gm}_0(T_X[1])
      \simeq \oH^{\BM,\Gm}_{2\mrm{vd}}(T^*_X[-2])
    \end{equation}
    where $\mrm{vd}$ is the virtual dimension of $X$ and we omit the Tate twist for simplicity.
    Via deformation to the derived normal bundle $N_{X/\Spec(k)} \simeq T_X[1]$ (cf. \cite[\S 3]{KhanVirtual}, \cite[\S 2.1]{virloc}) one has a canonical $\Gm$-equivariant Borel--Moore class on $T_X[1]$ (which roughly speaking is the cycle class of the intrinsic normal cone).
    This corresponds under \eqref{eq:lock} to an element
    \begin{equation}\label{eq:vfchfp}
      [X]^\mrm{vir} \in \oH^{\BM,\Gm}_{2\mrm{vd}}(T^*_X[-2]),
    \end{equation}
    which we may regard as a ``generalized'' virtual fundamental class; when $X$ is quasi-smooth we have $\oH^{\BM,\Gm}_{2\mrm{vd}}(T^*_X[-2]) \simeq \oH^{\BM}_{2\mrm{vd}}(X)$ and \eqref{eq:vfchfp} recovers the usual virtual fundamental class \cite{KhanVirtual}.
    The existence of such a class, living on a natural bundle over $X$ and of the expected dimension, can be regarded as confirmation of an expectation stated by Ciocan-Fontantine and Kapranov in the introduction of \cite{CiocanFontanineKapranov}.
    If we work over $k=\bC$ and $X$ is $1$-Artin with cotangent complex of amplitude $[-1, 2]$ (homologically graded), then by Borisov--Joyce and Oh--Thomas \cite{BorisovJoyce,OhThomas} there is another canonical class in $\oH^{\BM,\Gm}_{2\mrm{vd}}(T^*_X[-2])$ constructed using the $(-2)$-shifted symplectic structure on $T^*_X[-2]$, and the two should coincide.\footnote{%
      Thanks to Tasuki Kinjo for helpful discussions about this.
    }
  \end{rem}


\section{Easy proofs}
\label{sec:easy}

\ssec{Zero bundle}
\label{ssec:easy/zero}

  We prove \propref{prop:zero}.

  The evaluation map $\ev_{0_S} : [0^\vee_S/\Gm] \times [0_S/\Gm] \to [\A^1_S/\Gm]$ factors as the composite
  \[ B\bG_{m,S} \times B\bG_{m,S} \xrightarrow{m} B\bG_{m,S} \xrightarrow{\quo{i}} [\A^1_S/\bG_{m,S}] \]
  where the first map $m$ is $(L, L') \mapsto L \otimes L'$ and $i$ is the zero section.
  By localization, one computes ${i^*} {j_{*}}(\un) \simeq {q_{!}}(\un)[1]$, where $q : S \to B\bG_{m,S}$.
  Note that we have a cartesian square
  \[ \begin{tikzcd}
    B\bG_{m,S} \ar{r}\ar{d}{(\id,o)}
    & S \ar{d}{q}
    \\
    B\bG_{m,S} \times B\bG_{m,S} \ar{r}{m}
    & B\bG_{m,S}
  \end{tikzcd} \]
  where the left-hand vertical map is $L \mapsto (L, L^\vee)$.
  Thus the kernel of $\FT_0$ is
  \[
    \sP_0 = \ev_{0_S}^* j_*(\un)
    \simeq m^* q_!(\un)[1]
    \simeq (\id,o)_!(\un)[1]
  \]
  and $\FT_0$ itself is given by
  \begin{align*}
    \FT_0(\sF)
    &\simeq \pr_{1,!}(\pr_2^*(\sF) \otimes (\id,o)_!(\un))[1][-1]\\
    &\simeq \pr_{1,!}(\id,o)_!(\id,o)^*\pr_2^*(\sF)\\
    &\simeq o^*(\sF)
  \end{align*}
  by the projection formula.

  For the remaining isomorphisms, note that $o$ is finite étale so that $o_* \simeq o_!$ and $o^* \simeq o^!$.
  Since $o$ is an involution, we also have $o^*o^* \simeq \id$, hence $o^* \simeq o_*$ and $o_! \simeq o^!$.

\ssec{Base change, part 1}
\label{ssec:basechange1}

  Let the notation be as in \propref{prop:funbase}.

  The base change property for $\quo{j}_*(\un)$ (\propref{prop:j_*(1)}) implies that there is a canonical isomorphism
  \begin{equation}
    f_{E^\vee \fibprod_S E}^* (\sP_E) \simeq \sP_{E'}
  \end{equation}
  where $f_{E^\vee \fibprod_S E} : E'^\vee \fibprod_{S'} E' \to E^\vee \fibprod_S E$ is the base change of $f$.
  The isomorphisms \eqref{eq:BC^*} and \eqref{eq:BC_!} follow immediately using the base change and projection formulas.

  The proofs of the remaining two isomorphisms will be done after proving involutivity.

\ssec{Functoriality, part 1}
\label{ssec:easy/fun}

  We prove the isomorphism \eqref{eq:funA} of \propref{prop:fun}.

  Let $\phi : E' \to E$ be a morphism of derived vector bundles.
  We have the commutative square
  \[ \begin{tikzcd}
    \quo{E^\vee} \times \quo{E'} \ar{r}{\id \times {\phi}}\ar{d}{{\phi^\vee} \times \id}
    & \quo{E^\vee} \times \quo{E} \ar{d}{\ev_E}
    \\
    \quo{E'^\vee} \times \quo{E'} \ar{r}{\ev_{E'}}
    & \quo{\A^1}
  \end{tikzcd} \]
  whence a canonical isomorphism
  \[ (\id \times {\phi})^*\sP_{E}^* \simeq ({\phi^\vee} \times \id)^*(\sP_{E'}). \]

  We use the following commutative diagram, where all squares are cartesian.
  \[ \begin{tikzcd}[matrix xscale=0.2]
    & & \quo{E'} \times \quo{E'} \ar[leftarrow,swap]{ld}{\phi^\vee\times\id}\ar{rd}{\id\times f} \ar[bend right=50,swap]{lldd}{\pr_2}\ar[bend left=50]{rrdd}{\pr_1} & &
    \\
    & \quo{E^\vee} \times \quo{E'} \ar[swap]{ld}{\pr_2}\ar{rd}{\id\times f} & & \quo{E'^\vee} \times \quo{E} \ar[leftarrow,swap]{ld}{\phi^\vee \times \id}\ar{rd}{\pr_1} &
    \\
    \quo{E'} \ar{rd}{\phi} & & \quo{E^\vee} \times \quo{E} \ar[swap]{ld}{\pr_2}\ar{rd}{\pr_1} & & \quo{E'^\vee} \ar[leftarrow,swap]{ld}{\phi^\vee}
    \\
    & \quo{E} & & \quo{E^\vee} & &
  \end{tikzcd} \]

  By the base change and projection formulas we have:
  \begin{align*}
    \FT_{E} {\phi_!}(\sF)[1]
    &\simeq \pr_{1,!} (\pr_2^*({\phi_!}\sF) \otimes \sP_{E})\\
    &\simeq \pr_{1,!} ((\id \times {\phi})_!(\pr_2^*\sF) \otimes \sP_{E})\\
    &\simeq \pr_{1,!}(\id \times {\phi})_! (\pr_2^*\sF \otimes (\id \times {\phi})^*(\sP_{E}))\\
    &\simeq \pr_{1,!} (\pr_2^*\sF \otimes ({\phi^\vee} \times \id)^*(\sP_{E'}))
  \end{align*}
  Similarly we have:
  \begin{align*}
    {\phi^{\vee,*}} \FT_{E'}(\sF)[1]
    &\simeq {\phi^{\vee,*}} \pr_{1,!} (\pr_2^*(\sF) \otimes \sP_{E'})\\
    &\simeq \pr_{1,!} (\id \times {\phi})_! ({\phi^\vee} \times \id)^* (\pr_2^*(\sF) \otimes \sP_{E'})\\
    &\simeq \pr_{1,!} ({\phi^\vee} \times \id)^* (\pr_2^*(\sF) \otimes \sP_{E'})\\
    &\simeq \pr_{1,!} (\pr_2^*(\sF) \otimes ({\phi^\vee} \times \id)^* \sP_{E'}).
  \end{align*}
  The claim follows.
  
\ssec{Base change for the twist, part 1}
\label{ssec:basechangetwist1}

  Given a morphism $f : S' \to S$ and $E' := E \fibprod_S S' \in \DVect(S')$, we show the isomorphisms
  \begin{equation*}
    f^* \sL^E \simeq \sL^{E'}
  \end{equation*}
  and
  \begin{equation*}
    f^*_E((-)\tw{E}) \simeq f_E^*(-)\tw{E'},
    \qquad
    f_{E,!}(-)\tw{E} \simeq f_{E,!}((-)\tw{E'}).
  \end{equation*}
  of \lemref{lem:twistbase}.
  The remaining parts of these statements will be proven in \ssecref{ssec:basechangetwist2}.

  By \eqref{eq:BC^*} we have:
  \begin{align*}
    f^*\sL^E
    &= f^* \pi_{E,!} \FT_{E^\vee}(\un)\\
    &\simeq \pi_{E',!} f_{E^\vee}^* \FT_{E^\vee}(\un)\\
    &\simeq \pi_{E',!} \FT_{E'^\vee}(\un)\\
    &= \sL^{E'}.
  \end{align*}

  From $f^* \sL^E \simeq \sL^{E'}$ we now deduce
  \begin{align*}
    f_E^*((-)\tw{E})
    &= f_E^*(-\otimes\pi_E^*(\sL^E))\\
    &\simeq f_E^*(-)\otimes f_E^*\pi_E^*(\sL^E)\\
    &\simeq f_E^*(-)\otimes \pi_{E'}^*f^*(\sL^E)\\
    &\simeq f_E^*(-)\otimes \pi_{E'}^*(\sL^{E'})\\
    &= f_E^*(-)\tw{E'}.
  \end{align*}

  Similarly, using the projection formula:
  \begin{align*}
    f_{E,!}(-)\tw{E}
    &= f_{E,!}(-) \otimes \pi_E^*(\sL^E)\\
    &\simeq f_{E,!}(- \otimes f_E^*\pi_E^*(\sL^E))\\
    &\simeq f_{E,!}(- \otimes \pi_{E'}^*f^*(\sL^E))\\
    &\simeq f_{E,!}(- \otimes \pi_{E'}^*(\sL^{E'}))\\
    &= f_{E,!}((-)\tw{E'}).
  \end{align*}


\section{Proof of involutivity assuming cosupport}
\label{sec:cosuppinvol}

In this section we prove \lemref{lem:cosuppinvol}.

\ssec{Kernel of the square}

  Consider the following commutative diagram:
  \[\begin{tikzcd}[matrix xscale=0.1, matrix yscale=1.3]
    && \quo{E} \fibprod_S \quo{E} \\
    && \quo{E} \fibprod_S \quo{E^\vee} \fibprod_S \quo{E}  \\
    & \quo{E} \fibprod_S \quo{E^\vee} && \quo{E^\vee} \fibprod_S \quo{E} \\
    \quo{E} && \quo{E^\vee} && \quo{E}
    \arrow["\pr_1", swap, from=3-2, to=4-1]
    \arrow["\pr_2", from=3-2, to=4-3]
    \arrow["\pr_1", swap, from=3-4, to=4-3]
    \arrow["\pr_2", from=3-4, to=4-5]
    \arrow["\pr_{12}", swap, from=2-3, to=3-2]
    \arrow["\pr_{23}", from=2-3, to=3-4]
    \arrow["\pr_{13}", from=2-3, to=1-3]
    \arrow["\pr_1", swap, bend right, from=1-3, to=4-1]
    \arrow["\pr_2", bend left, from=1-3, to=4-5]
  \end{tikzcd}\]
  A straightforward application of base change and projection formulas yields an identification
  \begin{equation}\label{eq:FT2}
    {\FT_{E^\vee}} \circ {\FT_{E}}(-)
    \simeq \pr_{1,!}(\pr_2^*(-) \otimes \sP'')[-2]
  \end{equation}
  where
  \[ \sP'' := \pr_{13,!}(\pr_{12}^*(\sP_{E^\vee}) \otimes \pr_{23}^*(\sP_E)) \in \D(\quo{E} \fibprod_S \quo{E}). \]

  Consider the $\Gm$-action scaling both coordinates of $E \fibprod_S E$.
  The two projections $\pr_1, \pr_2 : E \fibprod_S E \to E$ are $\Gm$-equivariant.
  We let $c : \quo{(E \fibprod_S E)} \to \quo{E} \fibprod_S \quo{E}$ denote the induced morphism $(\quo{\pr_1}, \quo{\pr_2})$.
  We denote by $e : E \fibprod_S E \to E$ the ``difference'' morphism, given informally by $(x,y) \mapsto x-y$; this is also $\Gm$-equivariant.

  \begin{lem}\label{lem:P''}
    For any $E \in \DVect(S)$, there is a canonical isomorphism
    \begin{equation*}
      \sP''
      \simeq c_! \quo{e}^* \FT_{E^\vee}(\un)[2].
    \end{equation*}
  \end{lem}

\ssec{Proof of Lemma~\ref{lem:P''}}

  Consider the morphism
  \[ \ev'' : \quo{E} \fibprod_S \quo{E^\vee} \fibprod_S \quo{E} \to \quo{\A^1_S} \fibprod_S \quo{\A^1_S} \]
  given on points by
  \begin{multline*}
    \big((L, x : L \to E), (L', \phi : L' \to E^\vee), (L'', y : L'' \to E)\big)\\
    \,\mapsto\,
    \big(
    ( L' \otimes L, L' \otimes L \xrightarrow{\phi \otimes x} E^\vee \fibprod_S E \xrightarrow{\ev} \A^1_S ),
    ( L' \otimes L'', L' \otimes L'' \xrightarrow{\phi \otimes y} E^\vee \otimes E \xrightarrow{\ev} \A^1_S )
    \big).
  \end{multline*}

  We have commutative squares
  \[ \begin{tikzcd}
    \quo{E} \fibprod_S \quo{E^\vee} \fibprod_S \quo{E} \ar{r}{\pr_{23}}\ar{d}{\ev''}
    & \quo{E} \fibprod_S \quo{E^\vee} \ar{d}{\ev}
    \\
    \quo{\A^1_S} \fibprod_S \quo{\A^1_S} \ar{r}{\pr_1}
    & \quo{\A^1_S},
  \end{tikzcd}
  \quad
  \begin{tikzcd}
    \quo{E} \fibprod_S \quo{E^\vee} \fibprod_S \quo{E} \ar{r}{\pr_{12}}\ar{d}{\ev''}
    & \quo{E} \fibprod_S \quo{E^\vee} \ar{d}{\ev'}
    \\
    \quo{\A^1_S} \fibprod_S \quo{\A^1_S} \ar{r}{\pr_2}
    & \quo{\A^1_S}.
  \end{tikzcd} \]
  This yields
  \begin{align}
    \sP''
    &\simeq \pr_{13,!} \ev''^*(\quo{j}_*(\un) \boxtimes_S \quo{j}_*(\un))\\
    &\simeq \pr_{13,!} \ev''^* c_! \quo{j}_{\Delta,*}(\un)[1]
  \end{align}
  where the second isomorphism comes from Lemmas~\ref{lem:kunn} and \ref{lem:pnpqnq}.

  Next observe that we have a commutative diagram
  \begin{equation}\label{eq:lanthana}
    \begin{tikzcd}[matrix xscale=0.5]
      \quo{S} \ar{r}{\quo{0}}\ar[phantom]{rd}{\square}
      & \quo{E}\ar[phantom]{rd}{\square}
      & \quo{E^\vee} \fibprod_S \quo{E} \ar[swap]{l}{\pr_2}\ar{r}{\ev}
      & \quo{\A^1_S} \ar[leftarrow]{r}{\quo{j}}\ar[phantom]{rd}{\square}
      & \quo{\Gm}\ar[leftarrow]{d}{\quo{e}}
      \\
      \quo{E} \ar{r}{\quo{\Delta}}\ar[swap]{u}{\quo{\pi_E}}\ar[equals]{d}
      & \quo{(E\fibprod_S E)} \ar[swap]{u}{\quo{e}}\ar{d}{c}\ar[phantom]{rd}{\square}
      & \quo{E^\vee} \fibprod_S \quo{(E\fibprod_S E)} \ar[swap]{l}{\pr_2}\ar{r}\ar[swap]{u}{\id\times\quo{e}}\ar{d}\ar[phantom]{rd}{\square}
      & \quo{\A^2_S}\ar[swap]{u}{\quo{e}}\ar{d}{c}\ar[leftarrow]{r}{\quo{j_\Delta}}
      & \quo{(\A^2\setminus\Delta)}
      \\
      \quo{E} \ar{r}{\Delta}
      & \quo{E}\fibprod_S \quo{E}
      & \quo{E}\fibprod_S\quo{E^\vee}\fibprod_S\quo{E} \ar[swap]{l}{\pr_{13}}\ar{r}{\ev''}
      & \quo{\A^1_S}\fibprod_S\quo{\A^1_S}
      &
    \end{tikzcd}
  \end{equation}
  where the indicated squares are homotopy cartesian.
  The notation is as follows:
  \begin{itemize}
    \item
    The two projections $\pr_1, \pr_2 : E \fibprod_S E \to E$ are $\Gm$-equivariant.
    We let $c : \quo{(E \fibprod_S E)} \to \quo{E} \fibprod_S \quo{E}$ denote the induced morphism $(\quo{\pr_1}, \quo{\pr_2})$.
    Similarly for $c : \quo{\A^2_S} \to \quo{\A^1_S} \fibprod_S \quo{\A^1_S}$.

    \item 
    $\Delta$ is the diagonal of $\quo{E}$ and $\quo{\Delta}$ is the quotient of the diagonal of $E$.

    \item
    $e : E \fibprod_S E \to E$ is the ``difference'' morphism, given informally by $(x,y) \mapsto x-y$.
    Similarly for $e : \A^2 \to \A^1$.
  \end{itemize}

  We thus get
  \begin{align*}
    \sP''
    &\simeq \pr_{13,!} \ev''^* c_! \quo{j}_{\Delta,*} e^*(\un)[1]\\
    &\simeq c_! \quo{e}^* \pr_{2,!} \ev_E^* \quo{j}_*(\un)[1]
  \end{align*}
  by base change formulas.

  Under the automorphism of $\quo{E}^\vee \fibprod_S \quo{E}$ which swaps the factors, the morphism $\ev_E : \quo{E}^\vee \fibprod_S \quo{E} \to \quo{\A^1_S}$ is identified with $\ev_{E^\vee}$ and the projection $\pr_2 : \quo{E}^\vee \fibprod_S \quo{E} \to \quo{E}$ is identified with $\pr_1 : \quo{E} \fibprod_S \quo{E}^\vee \to \quo{E}$.
  Thus by definition
  \[ \FT_{E^\vee}(\un)
    \simeq \pr_{2,!} \ev_E^* \quo{j}_*(\un)[-1]. \]
  Hence we conclude
  \begin{equation*}
    \sP'' \simeq c_! \quo{e}^* \FT_{E^\vee}(\un)[2].
  \end{equation*}
  as claimed.

\ssec{Proof of Lemma~\ref{lem:cosuppinvol}}

  Since $E$ satisfies the cosupport property, we have the canonical isomorphism
  \begin{equation*}
    0_{E,!}(\cL^{E})
    := 0_{E,!}0_E^! \FT_{E^\vee}(\un)
    \to \FT_{E^\vee}(\un).
  \end{equation*}
  Applying $c_! \quo{e}^*$ yields a canonical isomorphism
  \begin{equation*}
    c_! \quo{e}^* 0_{E,!}(\cL^{E})[2]
    \to c_! \quo{e}^* \FT_{E^\vee}(\un)[2]
    \simeq \sP''.
  \end{equation*}
  By base change, using the diagram \eqref{eq:lanthana} again, we obtain a canonical isomorphism
  \begin{equation*}
    \Delta_! {\pi_E^*}(\cL^{E})[2]
    \to \sP''.
  \end{equation*}

  Finally plugging this into \eqref{eq:FT2} yields
  \begin{align*}
    \FT_{E^\vee}\FT_E(\sF)
    &\simeq \pr_{1,!}(\pr_2^*(\sF) \otimes \sP'')[-2]\\
    &\simeq \pr_{1,!}(\pr_2^*(\sF) \otimes \Delta_! {\pi_E^*}(\cL^{E}))\\
    &\simeq \pr_{1,!}(\Delta_!(\Delta^*\pr_2^*(\sF) \otimes {\pi_E^*}(\cL^{E})))\\
    &\simeq \sF \otimes {\pi_E^*}(\cL^{E})
  \end{align*}
  where the second-to-last step is the projection formula.


\section{Co/support and involutivity for \texorpdfstring{$E \le 0$}{E<=0}}
\label{sec:supp}

In this section we fix $E \in \DVect(S)$ of amplitude $\le 0$.
We prove the support property for $E$ (which implies the cosupport property too) and deduce the optimal form of involutivity in this case.

\ssec{Proof of \propref{prop:suppinvol}}
\label{ssec:supp/proof}

  \sssec{Restriction to zero}

    We first compute the inverse image along $\quo{0}_E : \quo{S} \to \quo{E}$:
    \[
      \quo{0}_E^* \FT_{E^\vee}(\un)
      \simeq \FT_{0_S} (\quo{\pi}_{E^\vee,!} (\un))
      \simeq \FT_{0_S} (\un\vb{-E^\vee})
      \simeq \un\vb{-E^\vee},
    \]
    using functoriality as well as the isomorphisms $\quo{\pi}_{E^\vee}^! \simeq \quo{\pi}_{E^\vee}^*\vb{E^\vee}$ (Poincaré duality) and $\quo{\pi}_{E^\vee,!} \quo{\pi}_{E^\vee}^! \simeq \id$ (homotopy invariance) for the morphism $\quo{\pi}_{E^\vee} : \quo{E^\vee} \to \quo{S}$ (smooth because $E^{\vee}$ is of amplitude $\ge 0$).

  \sssec{Support and cosupport properties}

    Since $E$ is of amplitude $\le 0$, $0_E$ is a closed immersion and $0_{E,*} \simeq 0_{E,!}$.
    By localization, invertibility of either \eqref{eq:doctrinally} or \eqref{eq:urodelan} is equivalent to the assertion that $\FT_{E^\vee}(\un)$ is supported on the image of $\quo{0}_E : \quo{S} \to \quo{E}$.

    It will suffice to show that for every residue field $v : \Spec(\kappa) \to \quo{E}$ that factors through the complement of $0_E$, $v^* \FT_{E^\vee}(\un)$ vanishes.
    Without loss of generality, we may replace $S$ by $\Spec(\kappa)$ and show that for every nowhere zero section $s$ of $E \to S$, the inverse image along $s' : S \xrightarrow{s} E \twoheadrightarrow \quo{E}$ vanishes.

    We have the following commutative diagram:
    \[ \begin{tikzcd}
      S\ar[leftarrow]{r}{a_{E^\vee}}\ar{d}{s'}
      & \quo{E^\vee} \ar{r}{\quo{\ev_s}}\ar{d}{\id \times s'}
      & \quo{\A^1_S} \ar[equals]{d}
      \\
      \quo{E}\ar[leftarrow]{r}{\pr_2}
      & \quo{E^\vee}\times \quo{E} \ar{r}{\ev}
      & \quo{\A^1_S}
    \end{tikzcd} \]
    where the left-hand square is cartesian.
    Here $\quo{\ev_s}$ is the $\Gm$-quotient of the ``evaluation at $s$'' morphism $\ev_s : E^\vee \to \A^1_S$, and $a_{E^\vee}$ is the projection.
    Hence we have
    \begin{equation}\label{eq:quadrans}
      s'^* \pr_{2,!} \ev^*
      \simeq a_{E^\vee,!} (\id\times s')^* \ev^*
      \simeq a_{E^\vee,!} \quo{\ev_s^*}.
    \end{equation}
    Since $s$ is nowhere zero, $\ev_s$ is surjective and its fibre $F = \Fib(\ev_s : E^\vee \to \A^1_S)$ is of amplitude $\ge 0$.
    Since $S$ is affine, it follows that $\ev_s$ admits a section, hence can be identified with the projection $F \fibprod_S \A^1_S \to \A^1_S$.
    Using Poincaré duality and homotopy invariance for the latter we have
    \begin{align*}
      a_{E^\vee,!} \,\quo{\ev_s^*}
      &\simeq a_{!} \,\quo{\ev_{s,!}} \quo{\ev_s^*}\\
      &\simeq a_{!} \,\quo{\ev_{s,!}} \quo{\ev_s^!} \vb{-F}\\
      &\simeq a_{!} \vb{-F}
    \end{align*}
    where $a : \quo{\A^1_S} \to S$ is the projection.
    Combining with \eqref{eq:quadrans} we deduce
    \begin{equation}
      s'^* \FT_{E^\vee}(\un)
      := s'^* \pr_{2,!} \ev^*(\quo{j}_*(\un))
      \simeq a_{!}\,\quo{j}_*(\un) \vb{-F}.
    \end{equation}
    Since $a$ factors through $\quo{p} : \quo{\A^1_S} \to \quo{S}$, and $\quo{p}_!\,\quo{j}_* \simeq 0$ (\propref{prop:j_*(1)}), this vanishes.

\ssec{Proof of involutivity}

  At this point, combining \propref{prop:suppinvol} with \lemref{lem:cosuppinvol} yields the following form of involutivity:

  \begin{cor}[Involutivity]\label{cor:conninvol}
    Let $E \in \DVect(S)$ be of amplitude $\le 0$.
    Then there is a canonical isomorphism
    \begin{equation*}
      (-)\vb{-E^\vee} \to \FT_{E^\vee} \FT_{E}(-).
    \end{equation*}
  \end{cor}
  \begin{proof}
    By \lemref{lem:cosuppinvol}, the cosupport property for $E^\vee$ yields the canonical isomorphism
    \begin{equation*}
      (-) \otimes \pi_{E}^*(\sL^{E}) \to \FT_{E^\vee} \FT_{E}(-).
    \end{equation*}
    By \propref{prop:suppinvol}\itemref{item:journalize}, we have the canonical isomorphism
    \begin{equation*}
      \sL^{E}
      \simeq \pi_{E,!} 0_{E,*}(\un)\vb{-E^\vee}
      \simeq \pi_{E,!} 0_{E,!}(\un)\vb{-E^\vee}
      \simeq \un_S\vb{-E^\vee}.
    \end{equation*}
  \end{proof}


\section{Cosupport and involutivity for general \texorpdfstring{$E$}{E}}
\label{sec:cosupp}

Let $E \in \DVect(S)$ be an arbitrary derived vector bundle.
We will now prove \propref{prop:main/cosupp}, which implies involutivity (\thmref{thm:invol}) by \lemref{lem:cosuppinvol}.
We will also show that the twist $\sL^E = \un\tw{E}$ is $\otimes$-invertible in general (\lemref{lem:L}).

\ssec{Proof of Proposition~\ref{prop:main/cosupp}}
\label{ssec:cosupp/proof}

  Let $E \in \DVect(S)$ and let us show that the canonical morphism
  \begin{equation}\label{eq:ideoplastia}
    \counit : 0_{E^\vee,!} 0_{E^\vee}^! \FT_E(\un_E)
    \to \FT_E(\un_E)
  \end{equation}
  is invertible.
  Equivalently, $\FT_E(\un_E)$ lies in the essential image of the fully faithful functor $0_{E^\vee,!}$ (\corref{cor:chlorine}).

  Let $f : S' \to S$ be a smooth surjection and adopt the notation of \propref{prop:funbase}.
  Applying $f_{E^\vee}^*$ on the left to the morphism \eqref{eq:ideoplastia} yields, by the base change formula $\Ex^*_! : f_{E^\vee}^* 0_{E^\vee,!} \simeq 0_{E'^\vee,!} f^*$, the exchange isomorphism $\Ex^{*!} : f^* 0_{E^\vee}^! \simeq 0_{E'^\vee}^! f_{E^\vee}^*$, and \eqref{eq:BC^*}, a canonical morphism
  \begin{equation}
    0_{E'^\vee,!} 0_{E'^\vee}^! \FT_{E'/S'}(\un_{E'})
    \to \FT_{E'/S'}(\un_{E'})
  \end{equation}
  which one easily checks is identified with the counit.
  Thus the claim is local on $S$ and we may assume that $S$ is affine.
  In particular we may, by choosing a presentation of $E$, write $E$ as the fibre of a morphism $d : E_+ \to E_-$ where $E_+$ is of amplitude $\ge 0$ and $E_-$ is of amplitude $\le 0$ (e.g. $d$ is the morphism
  \begin{equation}
    \begin{tikzcd}
      \cdots \ar{r} & E_{1} \ar{r}\ar{d} & E_0 \ar{r}\ar{d} & 0 \ar{d} & \\
      & 0 \ar{r} & E_{-1} \ar{r} & E_{-2} \ar{r} & \cdots &
    \end{tikzcd}
  \end{equation}
  where $\cdots \to E_1 \to E_0 \to E_{-1} \to \cdots$ is some presentation of $E$).
  We have the cartesian squares
  \begin{equation}\label{eq:nieceship}
    \begin{tikzcd}
      E \ar{r}{i}\ar{d}{\pi_E}
      & E_+ \ar{d}{d}
      \\
      S \ar{r}{0_{E_-}}
      & E_-,
    \end{tikzcd}
    \qquad
    \begin{tikzcd}
      E_-^{\vee} \ar{r}{d^\vee}\ar{d}{\pi_{E_-^{\vee}}}
      & E_+^{\vee} \ar{d}{i^\vee}
      \\
      S \ar{r}{0_{E^\vee}}
      & E^\vee
    \end{tikzcd}
  \end{equation}
  where $i$ is a closed immersion (since $E_-$ is of amplitude $\le 0$) and its dual $i^\vee$ is a smooth surjection (it is a torsor under $E_-^{\vee}$).
  We may thus check \eqref{eq:ideoplastia} is invertible after applying $i^{\vee,*}$ on the left; by base change and invertibility of $\Ex^{*!}$ this reduces to show that the morphism
  \begin{equation}\label{eq:Pellaea}
    \counit : d^{\vee}_! d^{\vee,!} \FT_{E_+} (i_! \un_E)
    \to \FT_{E_+} (i_! \un_E)
  \end{equation}
  is invertible.
  We claim that $\FT_{E_+} (i_! \un_E)$ is in the essential image of $d^\vee_!$.
  Since $\unit : \un \to d^{\vee,!} d^\vee_!(\un)$ is invertible (\corref{cor:chlorine}), it will follow from adjunction identities that \eqref{eq:Pellaea} is invertible.

  Using $\Ex^{*,\FT}$, $\Ex^*_!$, and the computation of $\FT_{E_-}(\un)$ (\propref{prop:suppinvol}, recall $E_-$ is of amplitude $\le 0$), we first compute:
  \begin{align}
    \FT_{E_+^{\vee}} (d^\vee_! \un_{E_-^{\vee}})
    &\simeq d^* \FT_{E_-} (\un_{E_-})\\
    &\simeq d^* 0_{E_-^{\vee},!}(\un_{E_-^{\vee}}) \vb{-E_-^\vee}\\
    &\simeq i_! (\un_E)\vb{-E_-^\vee}.
  \end{align}
  Using involutivity for $E_+^{\vee}$ (\corref{cor:conninvol}), we deduce the (canonical) isomorphism
  \begin{equation}\label{eq:hackneyer}
    \FT_{E_+}(i_! \un_E)
    \simeq \FT_{E_+}\FT_{E_+^{\vee}} (d^\vee_! \un)\vb{E_-^\vee}\\
    \simeq d^\vee_!(\un)\vb{E_-^\vee}\vb{-E_+}.
  \end{equation}
  The claim follows.

\ssec{Proof of Lemma~\ref{lem:L}}
\label{ssec:cosupp/L}

  We prove that $\sL^E := \pi_{E,!} \FT_{E^\vee}(\un)$ is $\otimes$-invertible.
  This is equivalent to the assertion that the evaluation morphism
  \begin{equation}\label{eq:shiny}
    \ev : \sL^E \otimes \uHom(\sL^E, \un_S) \to \un_S
  \end{equation}
  is invertible.
  If $f : S' \to S$ is a smooth morphism, then we have the canonical isomorphism
  \begin{equation*}
    f^*\uHom(\sL^E, \un_S)
    \simeq \uHom(f^*\sL^E, f^*\un_{S'})
  \end{equation*}
  under which the $*$-inverse image of \eqref{eq:shiny} is identified with
  \begin{equation*}
    f^*\sL^E \otimes \uHom(f^*\sL^E, f^*\un_{S'}) \to \un_{S'}.
  \end{equation*}
  Let $E' := E \fibprod_S S'$ and let $f_{E^\vee} : E'^\vee \to E^\vee$ be the base change of $f$.
  By \ssecref{ssec:basechangetwist1} we have $f^*\sL^E \simeq \sL^{E'}$.
  This shows that the question of invertibility of \eqref{eq:shiny} is local on $S$.
  In particular, we may assume that $E$ admits a presentation so that there are cartesian squares
  \begin{equation}
    \begin{tikzcd}
      E \ar{r}{i}\ar{d}{\pi_E}
      & E_+ \ar{d}{d}
      \\
      S \ar{r}{0_{E_-}}
      & E_-,
    \end{tikzcd}
    \qquad
    \begin{tikzcd}
      E_-^{\vee} \ar{r}{d^\vee}\ar{d}{\pi_{E_-^{\vee}}}
      & E_+^{\vee} \ar{d}{i^\vee}
      \\
      S \ar{r}{0_{E^\vee}}
      & E^\vee
    \end{tikzcd}
  \end{equation}
  where $E_+$ is of amplitude $\ge 0$ and $E_-$ is of amplitude $\le 0$, as above in \eqref{eq:nieceship}.
  For simplicity we may as well show the claim for $E^\vee$ instead, i.e., that $\sL^{E^\vee} = \pi_{E^\vee,!} \FT_{E}(\un)$ is $\otimes$-invertible.
  After $*$-inverse image along the smooth surjection $\pi_{E_-^{\vee}} : E_-^{\vee} \to S$, we have:
  \begin{align*}
    \pi_{E_-^{\vee}}^* \sL^{E^\vee}
    &\simeq \pi_{E_-^{\vee}}^* 0_{E^\vee}^! \FT_{E}(\un)\tag{\corref{cor:contractder}}\\
    &\simeq d^{\vee,!} i^{\vee,*} \FT_{E}(\un)\tag{$\Ex^{*!}$}\\
    &\simeq d^{\vee,!} \FT_{E_+}(i_!\un)\tag{\ref{eq:funA}}\\
    &\simeq d^{\vee,!} d^{\vee}_! (\un) \vb{E_-^\vee}\vb{-E_+}\tag{$\star$}\\
    &\simeq \un\vb{E_-^\vee}\vb{-E_+}\tag{\corref{cor:chlorine}},
  \end{align*}
  where $(\star)$ is from the isomorphism $\FT_{E_+}(i_!\un) \simeq d^\vee_!(\un)\vb{E_-^\vee}\vb{-E_+}$ \eqref{eq:hackneyer}.
  This is $\otimes$-invertible, so we conclude.


\section{Base change and functoriality, part 2}
\label{sec:funct2}

Using involutivity, we conclude the proofs of Propositions~\ref{prop:funbase} and \ref{prop:fun}.
Specifically, we will use the fact that the functor
\begin{equation*}
  \FT_{E^\vee}(\bullet)\tw{-E}
\end{equation*}
is inverse to $\FT_E$ (\corref{cor:inverse}).

\ssec{Proof of~\texorpdfstring{\eqref{eq:BC_*}}{BC\_*}}

  Let the notation be as in \propref{prop:funbase}.
  Recall the isomorphism
  \begin{equation*}
    f^*_E((-)\tw{E}) \simeq f_E^*(-)\tw{E'}
  \end{equation*}
  of \lemref{lem:twistbase}, proven in \ssecref{ssec:basechangetwist1}.
  Passing to right adjoints, we have:
  \begin{equation}\label{eq:horsepond}
    f_{E,*}(-)\tw{-E} \simeq f_{E,*}((-)\tw{-E'})
  \end{equation}

  Recall also the $*$-base change formula \eqref{eq:BC^*}
  \begin{equation*}
    f_{E^\vee}^* \FT_E \simeq \FT_{E'} f_E^*
  \end{equation*}
  proven in Subsect.~\ref{ssec:basechange1}.
  Passing to right adjoints yields
  \begin{equation*}
    f_{E,*} (\FT_{E'^\vee}(-)\tw{-E'})
    \simeq \FT_{E^\vee}(f_{E^\vee,*}(-))\tw{-E}.
  \end{equation*}
  Pulling out the twist using \eqref{eq:horsepond}, we deduce
  \begin{equation*}
    f_{E,*} (\FT_{E'^\vee}(-))
    \simeq \FT_{E^\vee}(f_{E^\vee,*}(-)).
  \end{equation*}
  Equivalently, replacing $E$ by $E^\vee$ gives the formula
  \begin{equation}
    f_{E^\vee,*} (\FT_{E'}(-))
    \simeq \FT_{E}(f_{E,*}(-)).
  \end{equation}

\ssec{Proof of \texorpdfstring{\eqref{eq:BC^!}}{(BC\^!)}}

  This follows from \eqref{eq:BC_!} exactly as above.

\ssec{Proof of~\texorpdfstring{\eqref{eq:funB}}{(Fun!)}}

  Passing to right adjoints from \eqref{eq:funA}
  \begin{equation*}
    \Ex^{*,\FT} : \phi^{\vee,*} \circ {\FT_{E'}} \to {\FT_{E}} \circ \phi_!
  \end{equation*}
  yields the canonical isomorphism
  \begin{equation}\label{eq:lanced}
    \phi^! \circ {\FT_{E^\vee}\tw{-E^\vee}} \to {\FT_{E'^\vee}\tw{-E'^\vee}} \circ \phi^\vee_*.
  \end{equation}
  Equivalently, applying this to $\phi^\vee$ in place of $\phi$ we get the canonical isomorphism
  \begin{equation}
    \Ex^{!,\FT} : \phi^{\vee,!} \circ {\FT_{E'}\tw{-E'}} \to {\FT_{E}\tw{-E}} \circ \phi_*.
  \end{equation}

\ssec{Proof of~\texorpdfstring{\eqref{eq:funA'}}{(Fun*)}}

  Begin with the isomorphism \eqref{eq:lanced} above.
  Applying $\FT_{E'}$ on the left and $\FT_{E}$ on the right, we get:
  \begin{equation*}
    \Ex^{\FT,!} : {\FT_{E'}} \circ \phi^!
    \to \phi^\vee_* \circ {\FT_{E}}.
  \end{equation*}

\ssec{Proof of~\texorpdfstring{\eqref{eq:funB'}}{(Fun!)}}

  Begin with the isomorphism \eqref{eq:funA}:
  \begin{equation*}
    \Ex^{*,\FT} : \phi^{\vee,*} \circ {\FT_{E'}} \to {\FT_{E}} \circ \phi_!.
  \end{equation*}
  Applying $\FT_{E'}\tw{-E}$ on the left and $\FT_{E}\tw{-E}$ on the right, we get:
  \begin{equation*}
    \Ex^{\FT,*} : {\FT_{E'}\tw{-E'}} \circ \phi^*
    \to \phi^\vee_! \circ {\FT_{E}\tw{-E}}.
  \end{equation*}

\ssec{Proof of Lemma~\ref{lem:twistbase}}
\label{ssec:basechangetwist2}

  Let $f : S' \to S$ be a morphism and $E' := E \fibprod_S S' \in \DVect(S')$.
  We prove the remaining isomorphisms of \lemref{lem:twistbase}.
  We already have the isomorphisms
  \begin{equation}\label{eq:cankerwort}
    f^*_E((-)\tw{E}) \simeq f_E^*(-)\tw{E'},
    \qquad
    f_{E,!}(-)\tw{E} \simeq f_{E,!}((-)\tw{E'})
  \end{equation}
  by \ssecref{ssec:basechangetwist1}.
  
  \sssec{$*$-Push}
  The claim
  \begin{equation*}
    f_{E,*}(-)\tw{E} \simeq f_{E,*}((-)\tw{E'})
  \end{equation*}
  which is equivalent by passage to left adjoints to
  \begin{equation*}
    f_E^*(-)\vb{-E'} \simeq f_E^*((-)\tw{-E}).
  \end{equation*}
  Applying $\tw{-E'}$ to both sides, this is equivalent to
  \begin{equation*}
    f_E^*(-) \simeq f_E^*((-)\tw{-E})\tw{E}.
  \end{equation*}
  But this holds by \eqref{eq:cankerwort}.

  \sssec{$!$-Pull}
  This follows from the $!$-push version by the same argument as above.

\appendix

\section{The contraction lemma}
\label{sec:equivariant}

  The following is well-known in the case of a separated morphism of schemes.
  Thanks to L.~Mann and J.~Scholbach for pointing out to me the much more general statement in \cite[Thm.~C.5.3]{DrinfeldGaitsgoryCompact} (and that the same argument works for an arbitrary topological weave) which we restate below.
  (In the original draft, we gave an argument just for the case of derived vector bundles.)

  \begin{prop}[Contraction lemma]\label{prop:contract}
    Let $p : X \to S$ be a morphism of derived Artin stacks and $s : S \to X$ a section.
    Suppose there is an $\A^1$-homotopy $\A^1 \times X \to X$ between $\id_X$ and $s \circ p$, so that the diagram
    \[ \begin{tikzcd}[matrix yscale=0.1,matrix xscale=1.3]
      X \ar{rd}{i_0}\ar[bend left,swap]{rrd}{s\circ p}
      &
      &
      \\
      & X \times \A^1 \ar{r}
      & X
      \\
      X \ar[swap]{ru}{i_1}\ar[bend right]{rru}{\id_X}
      &
      &
    \end{tikzcd} \]
    commutes.
    Then the canonical morphisms
    \begin{equation*}
      p_* \xrightarrow{\unit} p_*s_*s^* \simeq s^*,
      \qquad
      s^! \simeq p_!s_!s^! \xrightarrow{\counit} p_!
    \end{equation*}
    are invertible on $\Gm$-equivariant sheaves.
  \end{prop}

  \begin{cor}\label{cor:contractder}
    For every derived Artin stack $X$ and every derived vector bundle $E$ over $X$, the natural transformations
    \begin{align*}
      &\pi_{E,*} \xrightarrow{\unit} \pi_{E,*}0_{E,*}0_E^* \simeq 0_E^*\\
      &0_E^! \simeq \pi_{E,!}0_{E,!}0_E^! \xrightarrow{\counit} \pi_{E,!}
    \end{align*}
    are invertible on $\Gm$-equivariant sheaves.
    In particular, the functors $\pi_E^*$, $\pi_E^!$, $0_{E,*}$, and $0_{E,!}$ are all fully faithful on $\Gm$-equivariant sheaves.
  \end{cor}

  \begin{proof}
    The first claim is a special case of \propref{prop:contract}.
    For every $\sF \in \D^{\Gm}(X)$ there is a commutative diagram
    \[ \begin{tikzcd}
      \sF \ar{r}{\mrm{unit}_{\pi_E}}\ar[equals]{rd}
      & \pi_{E,*}\pi_E^*(\sF) \ar{d}{\mrm{unit}_{0_E}}
      \\
      & 0_E^*\pi_E^*(\sF)
    \end{tikzcd} \]
    where the vertical arrow is invertible by the first claim.
    This shows that $\unit : \id \to \pi_{E,*} \pi_E^*$ is invertible on $\Gm$-equivariant sheaves.
    Similarly, on $\Gm$-equivariant sheaves, the counit $\pi_{E,!}\pi_E^! \to \id$ is identified with the tautological isomorphism $0^!\pi^! \simeq \id$; the counit $0_E^*0_{E,*} \to \id$ is identified with $\pi_{E,*} 0_{E,*} \simeq \id$; and the unit $\id \to 0_{E}^! 0_{E,!}$ is identified with $\id \simeq \pi_{E,!} 0_{E,!}$.
  \end{proof}

  \begin{cor}\label{cor:chlorine}
    Let $E \in \DVect(S)$.
    For any cartesian square
    \begin{equation*}
      \begin{tikzcd}
        X_0 \ar{r}{i}\ar{d}{f_0}
        & X \ar{d}{f}
        \\
        S \ar{r}{0_E}
        & E
      \end{tikzcd}
    \end{equation*}
    where $f$ is smooth, the unit $\un_X \to i^!i_!(\un_X)$ is invertible.
  \end{cor}
  \begin{proof}
    Apply $f_0^*$ on the left to the isomorphism $\unit : \id \to 0_E^! 0_{E,!}$ (\corref{cor:contractder}).
    Under the isomorphisms $\Ex^*_!$ and $\Ex^{*!}$ (the latter since $f$ is smooth), the result is identified with $\unit : f_0^* \to i^! i_! f_0^*$.
  \end{proof}


\section{Computations on \texorpdfstring{$\quo{\A^1}$}{A1/Gm} and \texorpdfstring{$\quo{\A^1 \times \quo{\A^1}}$}{A1/Gm x A1/Gm}}
\label{sec:j}

We lift some simple computations from \cite{Laumon} (namely, Lemmas~1.4, 3.2, 3.3, and 3.4 of \emph{op. cit.}) to the generality of topological weaves.

\ssec{The sheaf \texorpdfstring{$\quo{j}_*(\un)$}{j\_*(1)}}

  We adopt the following notation:
  \begin{equation*}
    \begin{tikzcd}
      S \ar{r}{i}\ar[equals]{rd}
      & \A^1_S \ar[leftarrow]{r}{j}\ar{d}{p}
      & \bG_{m,S} \ar{ld}{q}
      \\
      & S &
    \end{tikzcd}
  \end{equation*}
  where $i$ is the zero section and $j$ is its complement.
  We consider the $\Gm$-scaling quotient of this whole diagram, writing the resulting morphisms as $\quo{i} : \quo{S} \to \quo{\A^1_S}$, etc.

  We record some basic observations about the sheaf $j_*(\un) \in \D^{\Gm}(\A^1_S)$, or more precisely
  \begin{equation*}
    \quo{j}_*(\un) \in \D(\quo{\A^1_S})
  \end{equation*}
  where $\quo{j} : S = \quo{\bG_{m,S}} \hook \quo{\A^1_S}$.

  \begin{prop}\label{prop:j_*(1)}\leavevmode
    \begin{thmlist}
      \item
      \emph{Constructibility.}
      The sheaf $\quo{j}_*(\un)$ is constructible.

      \item
      \emph{Base change.}
      For any morphism $f : S' \to S$, let $j' : S' = \quo{\bG_{m,S'}} \hook \quo{\A^1_{S'}}$ denote the base change of $j$ along $f' : \quo{\A^1_{S'}} \to \quo{\A^1_S}$.
      Then the canonical morphism
      \begin{equation*}
        \Ex^*_* : f'^* \quo{j}_*(\un) \to j'_*(\un)
      \end{equation*}
      is invertible.

      \item\label{item:retolerate}
      There is an exact triangle
      \[
        \quo{j}_!
        \to \quo{j}_*
        \to (\quo{i} \circ \quo{q})_![1],
      \]
      and we have $\quo{p}_!\,\quo{j}_* \simeq 0$ in $\D(\quo{S})$.

      \item\label{item:binh}
      There is a canonical isomorphism
      \[ \quo{j}_*(\un) \simeq u_! j_{1,*}(\un)[1] \]
      in $\D(\quo{\A^1})$, where $j_1 : \A^1 \setminus \{1\} \to \A^1$ is the complement of the unit section and $u : \A^1_S \twoheadrightarrow \quo{\A^1_S}$ is the quotient morphism.
    \end{thmlist}
  \end{prop}
  \begin{proof}
  We have the localization triangle
  \begin{equation}\label{eq:harmonica}
    \quo{j}_!
    \simeq \quo{j}_! \quo{j}^* \quo{j}_*
    \xrightarrow{\counit} \quo{j}_*
    \xrightarrow{\unit} \quo{i}_* \quo{i}^* \quo{j}_*.
  \end{equation}
  Applying $\quo{p}_!$ yields
  \begin{equation*}
    \quo{q}_!
    \to \quo{p}_! \quo{j}_*
    \xrightarrow{\unit} \quo{p}_! \quo{i}_* \quo{i}^* \quo{j}_*
    \simeq \quo{i}^* \quo{j}_*.
  \end{equation*}
  We have $\quo{p}_! \quo{j}_* \simeq \quo{i}^! \quo{j}_* \simeq 0$ by the contraction lemma (\propref{prop:contract}) and base change formula.
  We deduce a canonical isomorphism
  \begin{equation*}
    \quo{i}^* \quo{j}_* 
    \simeq \quo{q}_![1].
  \end{equation*}
  In particular, we can rewrite \eqref{eq:harmonica} as an exact triangle
  \begin{equation*}
    \quo{j}_!
    \to \quo{j}_*
    \to \quo{i}_* \quo{q}_![1]
    \simeq \quo{i}_! \quo{q}_![1].
  \end{equation*}

  Since $\quo{j}_!$, $\quo{i}_* = \quo{i}_!$, and $\quo{q}_!$ preserve constructible objects, it follows that $\quo{j}_!$ preserves constructible objects.
  Similarly, the base change formulas for $\quo{j}_!$, $\quo{i}_!$, and $\quo{q}_!$ yield the claimed base change formulas for $\quo{j}_*(\un)$ (we omit verification of commutativity of some diagrams, expressing e.g. the compatibility of $\Ex^*_* : f'^* \quo{j_*}(\un) \to j'_*(\un)$ and $\Ex^*_! : j'_!(\un) \to f'^* \quo{j_!}(\un)$).

  For the final claim~\itemref{item:binh}, we begin by observing the canonical isomorphism
  \[ j^* u_! j_{1,*} (\un) \simeq \un[-1] \]
  using the base change formula $j^* u_! \simeq q_! j^*$ for the square
  \[ \begin{tikzcd}
    \bG_{m,S} \ar{r}{j}\ar{d}{q}
    & \A^1_S \ar{d}{u}
    \\
    S = \quo{\bG_{m,S}} \ar{r}{\quo{j}}
    & \quo{\A^1_S}
  \end{tikzcd} \]
  and the observation that
  \[ q_! j^* j_{1,*}(\un) \simeq \un[-1] \]
  which is a straightforward computation using the base change formula and localization.

  It will then suffice to show that the unit morphism
  \[ u_! j_{1,*} (\un) \to \quo{j}_*\quo{j}^*(u_! j_{1,*} (\un)) \simeq \quo{j}_*(\un) \]
  is invertible.
  By localization, this is equivalent to showing that $\quo{i}^! (u_! j_{1,*} (\un)) \simeq 0$.
  By the contraction lemma (\propref{prop:contract}), we have
  \[ \quo{i}^! (u_! j_{1,*} (\un)) \simeq \quo{p}_! u_! j_{1,*} (\un) \simeq \quo{q}_! p_! j_{1,*}(\un) = 0\]
  since $p_! j_{1,*}(\un) = 0 \in \D(S)$ by a straightforward localization argument.
  \end{proof}

\ssec{The square of \texorpdfstring{$\quo{j}_*(\un)$}{j\_*(1)}}

  \sssec{Künneth formula}

    \begin{lem}\label{lem:kunn}
      There is a canonical isomorphism
      \[ \quo{j}_*(\un) \boxtimes_S \quo{j}_*(\un) \simeq (\quo{j} \fibprod_S \quo{j})_*(\un) \]
      in $\D(\quo{\A^1_S} \fibprod_S \quo{\A^1_S})$.
    \end{lem}
    \begin{proof}
      Consider the commutative diagram
      \[\begin{tikzcd}
        \quo{\bG_{m,S}} \ar[leftarrow]{r}{\pr_1}\ar{d}{\quo{j}}
        & \quo{\bG_{m,S}} \fibprod_S \quo{\bG_{m,S}} \ar{r}{\pr_2}\ar{d}{h}
        & \quo{\bG_{m,S}} \ar{d}{\quo{j}}
        \\
        \quo{\A^1_S} \ar[leftarrow]{r}{\pr_1}
        & \quo{\A^1_S} \fibprod_S \quo{\A^1_S} \ar{r}{\pr_2}
        & \quo{\A^1_S}
      \end{tikzcd}\]
      where $h := \quo{j}\fibprod_S\quo{j}$.
      There is a canonical morphism $\quo{j}_*(\un) \boxtimes_S \quo{j}_*(\un) \to h_*(\un)$, see e.g. \cite[2.1.19]{JY21}, which fits in a commutative diagram
      \[\begin{tikzcd}
        \quo{j}_!(\un) \boxtimes \quo{j}_!(\un) \ar{r}{\alpha}\ar[equals]{d}
        & \quo{j}_*(\un) \boxtimes \quo{j}_*(\un)\ar{d}
        \\
        h_!(\un) \ar{r}{\beta}
        & h_*(\un)
      \end{tikzcd}\]
      where the left-hand vertical arrow is the K\"unneth isomorphism for $!$-direct image, see e.g. \cite[Lem.~2.2.3]{JY21}, and the horizontal arrows are induced by the ``forget supports'' transformations for $j$ and $h$.
      If $\sK$ and $\sL$ denote the cofibres of $\alpha$ and $\beta$, respectively, it will suffice to show that the induced morphism $\sK \to \sL$ is invertible.

      The morphism $\alpha$ factors as the composite
      \[
        \quo{j}_!(\un) \boxtimes \quo{j}_!(\un)
        \xrightarrow{\alpha_1} \quo{j}_*(\un) \boxtimes \quo{j}_!(\un)
        \xrightarrow{\alpha_2} \quo{j}_*(\un) \boxtimes \quo{j}_*(\un).
      \]
      Computing the cofibres of $\alpha_1$ and $\alpha_2$ using \propref{prop:j_*(1)}\itemref{item:retolerate}, we get an exact triangle
      \begin{equation}\label{eq:cofalpha}
        (\quo{i} \circ \quo{q})_!(\un) \boxtimes \quo{j}_!(\un) \to \sK \to \quo{j}_*(\un) \boxtimes (\quo{i} \circ \quo{q})_!(\un).
      \end{equation}
      Similarly, since $h = (\id\times\quo{j})\circ(\quo{j}\times\id)$, $\beta$ factors as the composite
      \[
        (\id\times \quo{j})_!(\quo{j}\times \id)_!(\un)
        \xrightarrow{\beta_1} (\id\times \quo{j})_!(\quo{j}\times \id)_*(\un)
        \xrightarrow{\beta_2} (\id\times \quo{j})_*(\quo{j}\times \id)_*(\un).
      \]
      Computing the cofibres of $\beta_1$ and $\beta_2$ using \propref{prop:j_*(1)}\itemref{item:retolerate} (taking $S$ to be $\quo{\bG_{m,S}}$), we get the exact triangle
      \begin{equation}\label{eq:cofbeta}
        (\id\times \quo{j})_!(\quo{i}\circ\quo{q}\times\id)_!(\un)
        \to \sL
        \to (\id\times\quo{i}\circ\quo{q})_!(\quo{j}\times\id)_*(\un).
      \end{equation}

      The exact triangles \eqref{eq:cofalpha} and \eqref{eq:cofbeta} fit into a commutative diagram
      \[\begin{tikzcd}
        (\quo{i} \circ \quo{q})_!(\un) \boxtimes \quo{j}_!(\un) \ar{r}\ar[equals]{d}
        & \sK \ar{r}\ar{d}
        & \quo{j}_*(\un) \boxtimes (\quo{i} \circ \quo{q})_!(\un) \ar[equals]{d}
        \\
        (\id\times \quo{j})_!(\quo{i}\circ\quo{q}\times\id)_!(\un) \ar{r}
        & \sL \ar{r}
        & (\id\times\quo{i}\circ\quo{q})_!(\quo{j}\times\id)_*(\un).
      \end{tikzcd}\]
      The left-hand vertical arrow is the K\"unneth isomorphism for $!$-direct image, see e.g. \cite[Lem.~2.2.3]{JY21}.
      The right-hand vertical arrow is the isomorphism
      \begin{align*}
        \quo{j}_*(\un) \boxtimes (\quo{i} \circ \quo{q})_!(\un)
        &\simeq (\quo{j}\times\id)_*(\un) \otimes (\id\times \quo{i} \circ \quo{q})_!(\un)\\
        &\simeq (\id\times\quo{i} \circ \quo{q})_!(\id\times\quo{i} \circ \quo{q})^*(\quo{j}\times\id)_*(\un)\\
        &\simeq (\id\times\quo{i} \circ \quo{q})_!(\quo{j}\times\id)_*(\un)
      \end{align*}
      coming from the base change and projection formulas.
      The last isomorphism comes from the base change formula of \propref{prop:j_*(1)}, which we apply by taking $S$ to be $\quo{\A^1_S}$, $S'$ to be $\quo{\bG_{m,S}}$, and $f = \quo{i}\circ\quo{q}$.

      This shows that $\sK \to \sL$ is an isomorphism, as desired.
    \end{proof}

    \begin{rem}
      When (i) $S$ is locally noetherian, (ii) $S$ is defined over a base field $k$ which either admits resolution of singularities or whose characteristic is inverted in the weave $\D$, and (iii) $\D$ is continuous and constructibly generated, then \lemref{lem:kunn} can alternatively be deduced as follows.
      Indeed, by assumptions (i) and (iii), \cite[Prop.~3.13]{Weaves} says that it is enough to show the canonical morphism is invertible after $*$-inverse image to every residue field of $S$.
      As $\quo{j}_*(\un)$ is stable under base change in $S$ (\propref{prop:j_*(1)}), we may thus assume that $S$ is the spectrum of a field.
      In that case, the claim is a special case of the K\"unneth formula of \cite[Prop.~2.1.20]{JY21}, proven under assumption (ii).
    \end{rem}

  \sssec{}

  Consider the morphism
  \[ c = (\quo{\pr_1}, \quo{\pr_2}) : \quo{\A^2_S} \to \quo{\A^1_S} \fibprod_S \quo{\A^1_S} \]
  induced by the projections $\A^2 \to \A^1$ (which are $\Gm$-equivariant), which exhibits $\quo{\A^1_S} \fibprod_S \quo{\A^1_S}$ as the quotient of $\quo{\A^2_S}$ by the action $\lambda\cdot(x,y) = (x,\lambda\cdot y)$.

  Let $\Delta \subset \A^2_S$ denote the diagonal, $i_\Delta : \Delta \hook \A^2_S$ the inclusion, and $j_\Delta$ the complement.

  \begin{lem}\label{lem:pnpqnq}
    There is a canonical isomorphism
    \[ (\quo{j} \times \quo{j})_*(\un) \simeq c_! \quo{j}_{\Delta,*}(\un)[1] \]
  \end{lem}

  \begin{rem}
    Let $e : \A^2_S \to \A^1_S$ denote the ``difference'' morphism, given informally by $(x,y) \mapsto x-y$.
    By smooth base change for the square
    \[ \begin{tikzcd}
      \quo{(\A^2_S\setminus\Delta)} \ar{r}{\quo{j_\Delta}}\ar{d}{\quo{e}}
      & \quo{\A^2_S} \ar{d}{\quo{e}}
      \\
      \quo{\bG_{m,S}} \ar{r}{\quo{j}}
      & \quo{\A^1_S}
    \end{tikzcd} \]
    we can write $\quo{e}^* \quo{j}_* (\un) \simeq \quo{j}_{\Delta,*}(\un)$ and hence also
    \begin{equation}
      (\quo{j} \times \quo{j})_*(\un) \simeq c_! \quo{e}^* \quo{j}_* (\un)[1].
    \end{equation}
  \end{rem}

  \begin{proof}[Proof of \lemref{lem:pnpqnq}]
    Consider the commutative diagram
    \[ \begin{tikzcd}
      \bG_{m,S} \fibprod_S \bG_{m,S} \ar{r}\ar[twoheadrightarrow]{d}
      & \A^2_S \setminus \{0\}_S \ar{r}{j^2}\ar[twoheadrightarrow]{d}
      & \A^2_S \ar[twoheadrightarrow]{d}\ar[leftarrow]{r}{i^2}
      & S\ar[twoheadrightarrow]{d}
      \\
      \{1\}_S \fibprod_S \bG_{m,S} \ar{r}\ar[twoheadrightarrow]{d}
      & \quo{(\A^2_S \setminus \{0\}_S)} \ar{r}{\quo{j^2}}\ar[twoheadrightarrow]{d}
      & \quo{\A^2_S} \ar[twoheadrightarrow]{d}{c}\ar[leftarrow]{r}{\quo{i^2}}
      & \quo{S}\ar[twoheadrightarrow]{d}{d}
      \\
      S \ar{r}{j_a}
      & U \ar{r}{j_b}
      & \quo{\A^1_S} \fibprod_S \quo{\A^1_S}\ar[leftarrow]{r}{i_b}
      & \quo{S} \fibprod_S \quo{S}
    \end{tikzcd} \]
    where the bottom row is the quotient of the middle one by the $\Gm$-action which scales the second coordinate (with weight $1$).
    In the two left-hand columns, the horizontal rows are factorizations of $j \times j$, $\quo{(j \times j)}$, and $\quo{j} \times \quo{j}$, respectively.
    In the two right-hand columns, the horizontal rows are complementary open/closed immersions.

    Set
    $$\sF := c_!\quo{j_{\Delta,*}}(\un) \in \D(\quo{\A^1_S}\fibprod_S\quo{\A^1_S}).$$
    We claim there are isomorphisms:
    \begin{enumerate}
      \item\label{item:aouuhp/i} $i_b^!(\sF) \simeq 0$,
      \item\label{item:aouuhp/j} $j_b^*(\sF) \simeq {j}_{a,*}(\un)[-1]$.
    \end{enumerate}
    By localization it will follow that the unit
    \[ \mrm{unit} : \sF \to j_{b,*}j_b^*(\sF) \simeq j_{b,*}j_{a,*}(\un)[-1] \simeq (j \times j)_*(\un)[-1] \]
    is invertible, as claimed.
    
    \emph{Proof of \itemref{item:aouuhp/i}}.
    Since $\quo{i^2}$ is the zero section of the vector bundle $\quo{p^2} : \quo{\A^2_S} \to \quo{S}$ (where $p^2 : \A^2_S \to S$ is the projection), the contraction lemma (\propref{prop:contract}) yields an isomorphism
    \[ \quo{i^{2,!}} \simeq \quo{p_{0,!}^2} \]
    and similarly
    \[ i_b^! \simeq (\quo{p} \times \quo{p})_!. \]
    We get
    \begin{align*}
      i_b^! c_!\, \quo{j_{\Delta,*}}
      &\simeq (\quo{p} \times \quo{p})_! c_! \quo{j_{\Delta,*}}\\
      &\simeq d_!\, \quo{p_{0,!}^2}\, \quo{j_{\Delta,*}}\\
      &\simeq d_!\, \quo{i^{2,!}}\, \quo{j_{\Delta,*}}
    \end{align*}
    where $d$ is the diagonal of $\quo{S}$ as in the diagram above.
    But $\quo{i^{2,!}} \quo{j_{\Delta,*}} \simeq 0$ by base change (as $0 \in \A^2$ is contained in $\Delta$).

    \emph{Proof of \itemref{item:aouuhp/j}}.
    Consider the diagram of cartesian squares
    \[ \begin{tikzcd}[matrix xscale=0.3]
      \Gm\setminus\{1\}\ar[equals]{r}\ar[hookrightarrow]{d}{j'_1}
      & \quo{(\Gm \times \Gm\setminus\Delta)} \ar{r}\ar[hookrightarrow]{d}
      & \quo{(\A^1 \times \Gm\setminus\Delta)} \ar[equals]{r}\ar[hookrightarrow]{d}
      & \A^1 \setminus \{1\} \ar{r}\ar[hookrightarrow]{d}{j_1}
      & \quo{(\A^2\setminus\Delta)} \ar[hookrightarrow]{d}{\quo{j_\Delta}}
      \\
      \Gm \ar[equals]{r}\ar{d}{q}
      & \quo{(\Gm \times \Gm)} \ar{r}\ar[twoheadrightarrow]{d}
      & \quo{(\A^1 \times \Gm)} \ar[equals]{r}\ar[twoheadrightarrow]{d}
      & \A^1 \ar{r}\ar[twoheadrightarrow]{d}{f}
      & \quo{\A^2} \ar[twoheadrightarrow]{d}{c}
      \\
      S \ar[equals]{r}
      & \quo{\Gm} \ar[hookrightarrow]{r}
      & \quo{\A^1} \ar[equals]{r}
      & \quo{\A^1} \ar[hookrightarrow]{r}
      & \quo{\A^1} \times \quo{\A^1}
    \end{tikzcd} \]
    where we omit the subscripts $_S$ for simplicity.
    By base change we get
    \[
      (\quo{j} \times \quo{j})^* c_! \,\quo{j_{\Delta,*}} (\un)
      \simeq q_{0,!} j'_{1,*} (\un)
      \simeq \un[-1]
    \]
    where the second isomorphism follows easily from localization.

    Since $\quo{j} \times \quo{j} = j_b \circ j_a$ this isomorphism gives by adjunction a morphism
    \[ j_b^*(\sF) \to j_{a,*}(\un)[-1] \]
    in $\D(U)$ which we claim is invertible.
    Write $U$ as the union of the two opens $[\quo{(\A^1\times\Gm)}/\Gm]$ and $[\quo{(\Gm\times\A^1)}/\Gm]$, where $[-/\Gm]$ is the quotient by the scaling action on the second coordinate.
    Over either open this morphism restricts to the isomorphism
    \[ u_! j_{1,*}(\un) \simeq \quo{j}_*(\un)[-1] \]
    constructed in \propref{prop:j_*(1)}.
  \end{proof}



\bibliographystyle{halphanum}

Institute of Mathematics,
Academia Sinica,
Taipei 10617,
Taiwan

\end{document}